\documentclass{amsart}
\usepackage{amssymb}
\usepackage{amsmath,amsfonts,amsthm,amstext}
\usepackage[all]{xy}
\newtheorem{theorem}{Theorem}[section]
\newtheorem{lemma}[theorem]{Lemma}
\newtheorem{prop}[theorem]{Proposition}

\numberwithin{equation}{section}

\begin{document}

\title{On self-similar Lie algebras and virtual endomorphisms}

\author{Vyacheslav Futorny, Dessislava H. Kochloukova, Said N. Sidki}

\address{University of S\~ao Paulo (USP), S\~ao Paulo, Brazil; State University of Campinas (UNICAMP), Campinas, Brazil; Universidade de Bras\'ilia (UnB), Bras\'ilia, Brazil} 
\email{}

\subjclass[2000]{Primary  Secondary }

\date{}

\keywords{}

\begin{abstract} We introduce the notion of virtual endomorphisms of Lie algebras and use it 
as an approach for constructing self-similarity of Lie algebras. This is done in particular 
for a class of metabelian Lie algebras having homological type $FP_n$, which are variants of lamp-lighter groups. We establish several criteria when the existence of  virtual endomorphism implies a self-similar Lie structure.
Furthermore, we prove that the classical Lie algebra $sl_n(k)$, where $char(k)$ does not divide $n$
affords non-trivial faithful self-similarity.
 
\end{abstract}

\maketitle

\section{Introduction} \label{intr} Let $X$ be a commutative $k$-algebra with $1$, ${\mathcal Der} X$ be the Lie algebra of derivations of $X$ and  $L$ be a Lie algebra over $k$. Consider the wreath product of Lie algebras
$$  L \wr {\it Der } X : = (X \otimes_k L)  \leftthreetimes {\mathcal Der} X;$$
the wreath Lie product is natural and its  definition is given explicitly in Section \ref{section-structure}. 
Such  Lie algebras $ L \wr {\mathcal Der } X $ are important objects in the representation theory of Lie algebras.
In the case when $X=\mathbb C[t, t^{-1}]$, the derivation algebra ${\mathcal Der }   X$ is the centerless Virasoro algebra. Extending ${\mathcal Der }   X$ by a one dimensional center leads to the construction known as the Affine-Virasoro algebra
studied by many authors (e.g. \cite{Kac}, \cite{Ku}). These algebras are related to the conformal field theory, 
as the even part of the $N = 3$ superconformal algebra is the Affine-Virasoro
algebra with $L=sl(2)$.  The representation theory of Affine-Virasoro algebras was developed in  \cite{GHL}, \cite{GLZ}, \cite{HX}, \cite{JY}, \cite{LQ}, \cite{MRY} among  others.  Generalizing $X$ to the algebra of polynomial functions on $n$-dimensional torus the Lie algebra $ L \wr {\mathcal Der } X $   becomes the \emph{full toroidal} algebra. It was first introduced by Moody, Rao and Yokonuma in \cite{MRY}. The representation theory of these algebras was studied in   \cite{BB}, \cite{B}, \cite{R1}, \cite{R2} among  others.

The $n$-point Lie algebras, studied by Bremner \cite{Br1}, \cite{Br2} correspond to the case
$X=\mathbb C [t, (t -a_1)^{-1},  \ldots, (t - a_N)^{- 1}]$  with $a_1, \ldots, a_N$ distinct complex numbers. There is a vast literature  about the structure, central extensions and representations of these algebras,  
cf.  \cite{BCF}, \cite{C} and \cite{CJ} and references therein.
These algebras are particular cases of Krichever-Novikov algebras $L\otimes X$ studied in  \cite{KN1}, \cite{KN2}   in  connection  with the string theory in Minkowski space, where $X$ is the algebra of meromorphic functions on a Riemann surface of any genus with a finite number of poles.

On the other hand, there is a growing interest in the representation theory of the derivation algebras ${\mathcal Der }   X$.  
If $X=\mathcal F(Y)$ is the algebra of polynomial functions on an algebraic affine variety $Y$ then the derivation algebras ${\mathcal Der }   X$ is a source of simple Lie algebras. 
In fact, ${\mathcal Der }   X$ is  simple  if and only if $Y$ is a smooth variety \cite{S}.  
The classical results of Mathieu \cite{M1} on the representations of the centerless Virasoro algebra (the first Witt algebra) were recently generalized to an arbitrary torus ($n$th Witt algebra) in \cite{BF1}. The structure of the induced modules over ${\mathcal Der }   X$ for the  $n$-dimensional torus were studied in \cite{BF2}.  For a general irreducible affine variety $Y$ 
representation theory of the derivation
algebra ${\mathcal Der }   \mathcal F(Y)$ was initiated in \cite{BFN} and  \cite{BN}. 

Because of the strong connections between the loop algebra $L\otimes X$ and the derivation algebra ${\mathcal Der }   X$ and between their representations it is quite natural to 
combine these two structures together  in  $ L \wr {\mathcal Der } X $.

In \cite{Bartholdi} Bartholdi defined self-similar Lie algebra $L$ over a field $k$ as a Lie $k$-algebra  endowed with a monomorphism
  \begin{equation} \label{def01} 
  \psi : L \to L \wr {\it Der } X.
  \end{equation}
  We say that 
  $L$ is  a {\bf $\psi$-self-similar} Lie algebra if there is a monomorphism as in (\ref{def01}), where $\psi$ is proper, in the sense that it does not map $L$ into some Lie subalgebra of $X \otimes_k L$  or into ${\mathcal Der} X$.
  If we do not want to emphasize  the map $\psi$, we refer simply to  $L$ as a self-similar Lie algebra.   
   In all examples considered in \cite{Bartholdi}, $X$ is a $k$-algebra of finite dimension but in the general definition this is not required. 
Throughout the paper, if not mentioned otherwise $\otimes$ denotes $\otimes_k$.

  The wreath product construction $ L \wr {\mathcal Der } X $ should be viewed as a Lie algebra version of the restricted  wreath product of groups
and $\psi$ as the Lie algebra version of a self-similar representation of a group in its action on a regular tree.  In analogy with the Grigorchuk 2-group and the Gupta-Sidki p-groups, which are self-similar, Petrogradsky, Shestakov and Zelmanov  constructed in \cite{Petr}, \cite{Petr-Shest}, \cite{Shest-Zel}  Lie algebras that are either self-similar or  embed in self-similar Lie algebras of finite characteristic as  ideals of finite codimension.  
It is worth mentioning that   there are several  constructions of wreath Lie algebras different from the one we adopt in this paper, see \cite{P-R-S}, \cite{Shm}, \cite{Sul}.
   
   Applying recursively $\psi$ we get
   $$\psi^m : L \to (X^{\otimes m}  \otimes L) \leftthreetimes {\mathcal Der} (X^{\otimes m})$$ and this induces a Lie algebra homomorphism $L \to {\mathcal Der} (X^{\otimes m})$ that will be explained in Section \ref{section-structure}. This gives a Lie algebra homomorphism $\nu : L \to End_k(T(X))$, where $T(X)$ is the tensor algebra $\oplus_{m \geq 0} X^{\otimes m}$. Suppose in addition to (\ref{def01}) the map $\nu$ is injective, then following Bartholdi
   we say
that   \begin{equation} \label{main-def} L \hbox{ is  a }  \hbox{ {\bf faithful} }\psi \hbox{-self-similar Lie algebra}.\end{equation}
   Again  to abbreviate notation we say that $L$ is a faithful self-similar Lie algebra. 
     It is easy to show examples of Lie algebras that are $\psi$-self similar but are not faithful $\psi$-self-similar. We are primary interested in faithful self-similar Lie algebras.

Drawing analogy from the theory of self-similar groups  we note that the vector space $ X^{\otimes m}$ plays the role of the $m$-th level of the homogeneous one rooted tree in the case of self-similar groups, so the condition that $\nu : L \to End_k(T(X))$ is injective  means that 
given a non-zero element of L, it has a non-trivial action on some level.
 When $X = k[x]/(x^p)$, where $char(k) = p > 0$, we say that a $\psi$-self-similar Lie algebra is transitive if for every $m \geq 1$, $ X^{\otimes m}$ is a cyclic left $U(L)$-module with generator $x^{p-1} \otimes  \ldots \otimes  x^{p-1}$, where $U(L)$ is the universal enveloping algebra of $L$. 

If we want to draw attention to the ring $X$ and the Lie algebra $L_0 = Im (\pi \psi)$, where $$\pi :  L \wr {\it Der } X \to  {\it Der } X \hbox{ is the canonical projection}$$ and $L$ satisfies  (\ref{def01}), we say that $L$  is a {\bf $(X, L_0)$-self-similar} Lie algebra.
         We say $L$ is a rank $m$ self-similar Lie algebra, 
 if $L_0$ is a Lie subalgebra of ${\mathcal Der} X$ of finite dimension $m$. The self-similar Lie algebras considered in \cite{Bartholdi} are all  of rank 1; are $(X,L_0)$-self-similar Lie algebras with $X = k[x]/ (x^p)$ and $L_0 = k \partial / \partial x$. These Lie algebras include the  
   examples of Petrogradsky, Shestakov and  Zelmanov  in \cite{Petr}, \cite{Petr-Shest}, \cite{Shest-Zel}.
 
A number of notions used in groups acting on trees have been carried over to the Lie algebra setting. Thus, following    
 Bartholdi we say that an element  $a$  of a $\psi$-self-similar Lie algebra  $L$ is {\bf finite state} if there is a finite dimensional  vector subspace $S$ of $L$ containing $a$ such that 
\begin{equation} \label{state}
\psi(S) \in (X \otimes S) \leftthreetimes {\mathcal Der} X.
  \end{equation}

    There is an extensive theory of self-similar groups. A link between virtual endomorphisms and self-similar groups was established by Nekrashevych, Sidki in  \cite{Nekra}, \cite{N-S}  and was later used by Berlatto, Dantas, Kochloukova, Sidki to  construct new classes of self-similar groups \cite{Ber-Sidki}, 
 \cite{Alex-Said}, \cite{Alex-Said2}, \cite{DesiSaid}. As in the theory of self-similar groups we develop  a theory of virtual endomorphisms of Lie algebras with applications to self-similar Lie algebras.
   In Section \ref{def-endo} we associate to every Lie algebra homomorphism 
   $ \psi : L \to L \wr {\it Der } X$
       a virtual endomorphism $\theta : H \to L$, where $H$ is an ideal of $L$. More precisely, given an augmentation map $\epsilon : X \to k$ (i.e. an epimorphism of $k$-algebras)  we define \begin{equation} \label{endo} \theta = (\epsilon \otimes id_L) \psi \hbox{ and }H =  Ker (\pi \psi).\end{equation} The map $\theta$ was only mentioned (but not studied) in \cite{Bartholdi}, as part of  the definition of a recurrent self-similar Lie algebra as a self-similar algebra for which $\theta$ is surjective. We call $\theta$ a virtual endomorphism since in the case when $X$ is finite dimensional, the ideal $H$ has finite codimension in $L$.  Although in the case of  groups the existence of virtual endomorphism implies the relevant group acts on a regular tree (but possibly with non-trivial kernel), the situation in the Lie algebra case  is not the same i.e given a virtual endomorphism $\theta$ it is not true in general that there is a monomorphism $\psi$ that satisfies (\ref{endo}). The main results of this paper areconcern with conditions on virtual endomorphisms of Lie algebras that imply the existence of a compatible self-similar  structure of the Lie algebra.   
   
Our first result shows that when $L$ is a $(X,L_0)$ self-similar Lie algebra and $L_0$ contains some special abelian Lie subalgebra then we can recover $\psi$; that is we can recover the self-similar structure from the virtual endomorphism $\theta$. 

{\bf Important notation} For a Lie algebra $L$ and $h,b \in L$ we denote $b^m \circ h$ for $ad(b)^m(h) = [b, ad(b)^{m-1} (h)]$ and $b^0 \circ h = h$.
  
\bigskip
{\bf Theorem A} {\it 
  Let $X = k[x_1, \ldots, x_n]/I$, where $I = 0$ if $char(k) = 0$ and $I = (x_1^p, \ldots, x^p)$ if $char(k) = p$.
Let $L$ be a Lie algebra over $k$  and
$$
  \psi : L \to L \wr {\it Der } X : = (X \otimes_{k} L)  \leftthreetimes {\mathcal Der} X
  $$ be     a Lie algebra homomorphism. 
  Let $\theta : H \to L$ be the endomorphism associated to $\psi$ with respect to the augmentation map $\epsilon : X \to k$ as in (\ref{endo}), where $\epsilon$ is defined by $\epsilon(x_i) = 0$ for every $1 \leq i \leq n$.

  Assume that 
  
  1. $L = H  \leftthreetimes L_0$; 
  
  2. $L_0$ contains an abelian Lie algebra with $k$-linear basis $y_1, \ldots , y_n$ such that $\psi(y_i) = \partial / \partial x_i$; 
  
  3. there is a positive integer $m$ such that $\theta(y_i^m \circ H) = 0$ for every $ 1 \leq i \leq n$ and if $char(k) = p > 0$ then $m = p$.

  Then 
  
  a) for every $h \in H$ 
\begin{equation} \label{main-eq}
\psi(h) =
 \sum_{0 \leq i_1, \ldots, i_n \leq m_0 = m-1} ( i_1 ! \ldots i_n !)^{-1} x_1^{i_1} \ldots x_n^{i_n}  \otimes 
   \theta(y_1^{i_1} \ldots y_n^{i_n} \circ h);
\end{equation}

b) $\psi$ is a monomorphism if and only if  $Ker(\theta)$ does not contain non-trivial ideals of $L$;  

c) $L$ is a faithful $\psi$-self-similar Lie algebra if and only if there is not a non-trivial $\theta$-invariant ideal $J$ of  $L$ such that $J \subseteq H$ i.e. $\theta(J) \subseteq J$;

d) an element $h \in H$ is finite state if there is a finite dimensional subspace $V$ of $H$ such that $h \in V$ and  $\theta(y_1^{i_1} \ldots y_n^{i_n} \circ V) \subseteq V$ for every $ 0 \leq i_1, \ldots, i_n \leq m-1$;

e) if $n = 1$, $char(k) = p > 0$ and $\theta(H) = L$ then $L$ is a transitive $\psi$-self-similar Lie algebra.}

\bigskip Using Theorem A we will create criteria for the situation when the existence of a virtual  endomorphism $\theta : H \to L$ implies the existence of self-similar map $\psi$. Not surprisingly, the easiest case is when $L_0$ is abelian, which we treat it  in  
Theorem B. 

   \medskip
   {\bf Theorem B} {\it Let $L$ be a Lie algebra over a field $k$, $H$ an ideal of $L$ such that  $L = H  \leftthreetimes L_0$, $L_0$ is abelian and 
$dim(L_0) = n < \infty$.   Let $$\theta : H \to L$$ be a Lie algebra homomorphism such that there is $m  \in {\mathbb N}$ such that for every $b \in L_0$ and $h \in H$ we have $\theta(b^m \circ h) = 0$. If $char(k) = p >0$ we assume that $m = p$. 
     
     Then
      there is a homomorphism of Lie algebras
      $$\psi :  L \to L \wr {\mathcal Der } X = (X \otimes L) \leftthreetimes {\mathcal Der}X
     $$
     given by (\ref{main-eq})    
 where     
      
      1. $X = k[x_1, \ldots, x_n]/I$, where  $I = (x_1^p, \ldots, x_n^p)$ if $char(k) = p >0$ and $I = 0$ if $char(k) = 0$;
      
      2. $L_0 $ has a basis $y_1, \ldots, y_n$ over $k$ such that $\psi(y_i) = \partial / \partial {x_i}$ for every $ 1 \leq i \leq n$;  
     
     3. $\theta$ is the virtual endomorphism associated to $\psi$ with respect to the $k$-algebra map $\epsilon : X \to k$ defined by $\epsilon(x_i) = 0$ for every $1 \leq i \leq n$.

}
   
   \medskip

   In the rest of the  paper we 
prove that several Lie algebras  $L = H   \leftthreetimes L_0$, with $L_0$ a Lie subalgebra of ${\mathcal Der}X$, to be self-similar for specific
virtual endomorphisms. We note that in all these Lie algebras $L_0$ contains the finite dimensional, abelian Lie subalgebra $D_0 $ of ${\mathcal Der}X$ that is spanned by the differentials $\{ \partial / \partial x_j \}_{1 \leq j \leq n}$. 
The algebras $L_0$ which appear in our constructions are
 : a copy of $sl_{n+1}(k)$, the Frank algebra defined in \cite{Frank}, the full Jacobson-Witt algebra $W(n, \underline{1})$ \cite{Stradebook}, the Heisenberg Lie algebra $\mathcal H$ of dimension 3 and its central extension ${\mathcal H} \oplus k$. 
   In the case when $L_0$ is a non-abelian  Lie algebra  of  
   ${\mathcal Der}X$  in the  list stated above, there are versions of Theorem B that are more technical and are given by Theorem \ref{sl-2}, Theorem \ref{Frank0}, Theorem \ref{sl}, Theorem \ref{heisenberg} and Theorem \ref{heisenberg2} in  Section \ref{sl2-Witt}, Section \ref{Frank-section} and Section \ref{section-heisenberg}. 
   
   In    Section \ref{embed-new} we show a surprising self-similar structure on $sl_{n+1}(k)$, 
when $char(k)$ does not divide $n+1$.
		We know that  $sl_{n+1}( k)$ embeds in ${\mathcal Der} X$, where $X = k[x_1, \ldots, x_n] / (x_1^p, \ldots, x_n^p)$ if $p > 0$  and $X = k[x_1, \ldots, x_n]$ if $p = 0$.  
In ``matrix'' notation  i.e. $E_{a,b} E_{c,d} = \delta_{b,c} E_{a,d}$ we have :  
$$
E_{i,j} = - x_j \partial / \partial x_i \hbox{ for } 1 \leq i,j \leq n;$$ $$
E_{n+1,i} = - x_i \sum_{1 \leq j \leq n} x_j \partial / \partial x_j \hbox{ for } 1 \leq i \leq n;$$ $$
E_{i,n+1} = \partial / \partial x_i \hbox{ for } 1 \leq i \leq n \hbox{ and }
E_{n+1, n+1} =  \sum_{1 \leq i \leq n} x_i \partial / \partial x_i.$$
 The $k$-span of 
$\{ E_{i,j} \ |  \ 1 \leq i \not= j \leq n+1 \} \cup \{ E_{i,i} - E_{i+1, i+1} \ | \ 1 \leq i \leq n \}$ is a copy of $sl_{n+1}(k)$ and coincides with the $k$-span of 
$\{ E_{i,j} \ |  \ 1 \leq i \not= j \leq n+1 \} \cup \{ E_{i,i} \ | \ 1 \leq i \leq n \}$ .

   \bigskip
 {\bf Theorem C} {\it Let $k$ be a field such that  $char(k)$ does not divide $n+1$.  Then
there is a Lie algebra monomorphism
$$\psi : sl_{n+1}(k) \to (X \otimes sl_{n+1}( k)) \leftthreetimes {\mathcal Der}(X)$$
		  where $X = k[x_1, \ldots, x_n]$ if $char(k) = 0$ and $X = k[x_1, \ldots, x_n]/(x_1^p, \ldots, x_n^p)$ if $char(k) = p > 0$,  given by 
		  $$\psi(E_{i, n+1}) = \partial / \partial x_i  \hbox{ for } 1 \leq i \leq n,$$
		  $$   \psi(E_{i,j}) = 1 \otimes E_{i,j} + E_{i,j}  \hbox{ for } 1 \leq i,j \leq n,$$
		  and
		  $$\psi(E_{n+1,i}) = \sum_{1 \leq j \leq n} x_j \otimes b_{i,j} + 1 \otimes E_{n+1,i} + E_{n+1,i}  \hbox{ for } 1 \leq i \leq n,$$ where $$b_{i,j} = E_{j,i} - \delta_{i,j} E_{n+1, n+1} = E_{j,i} + \delta_{i,j}  \sum_{1 \leq i \leq n} E_{i,i}.$$
		 Furthermore $sl_{n+1}(k)$ is a faithful $\psi$-self-similar Lie algebra. }
   
   \bigskip
In Section \ref{abelian} we give examples of abelian and nilpotent faithful self-similar Lie algebras, in particular we show that  every countable dimensional abelian Lie algebra and the Heisenberg Lie algebra are  faithful self-similar. Furthermore we consider a nilpotent Lie algebra inside $gl_3(\mathbb{F}_p[x])$ and show it is  faithful self-similar.  

In Section \ref{fpn} we consider an example of a metabelian faithful self-similar Lie algebra.
In \cite{Br-Gr1}, \cite{Br-Gr2} Bryant and Groves established a criterion when a finitely generated, metabelian Lie algebra $L$ is finitely presented. 
 In \cite{Desi} Kochloukova gave   a criterion when a metabelian Lie algebra $L$ is of homological type $FP_m$ provided the extension is split i.e. $L = A \leftthreetimes Q$ with $A$ and $Q$ abelian. Recall that a Lie algebra $L$ over a field $k$ is of homological type $FP_m$ if there is a projective resolution of the trivial $U(L)$-module $k$, where all projectives are finitely generated in dimension $ \leq m$   and $U(L)$ is the enveloping algebra of $L$. 
Note that little is known for soluble ( but not metabelian) finitely presented Lie algebras. As shown by Wasserman in \cite{Wasserman} for a soluble finitely presented Lie algebra $L$ every ideal $I$ of codimension 1 is finitely generated as a Lie algebra. If  a soluble Lie algebra $L$ is of homological type $FP_{\infty}$ ( i.e. $L$ is $FP_m$ for every $m \geq 1$) Groves and Kochloukova showed in \cite{D-J} that $L$ is  finite dimensional.

Using the Bryant-Groves theory for metabelian Lie algebras and the main result from \cite{Desi}  we construct in Theorem D a self-similar metabelian Lie algebra of type $FP_n$. Note that by the Bryant-Groves theory of metabelian Lie algebras every metabelian Lie algebra of type $FP_2$  is  finitely presented. When  $n = 1$  and $char(k) = p > 0$, the Lie algebra $L$ in Theorem D is the lamplighter Lie algebra which we define in Section \ref{lamp}.

\medskip
{\bf Theorem D} {\it Let $Q$ be an abelian Lie algebra over a field $k$, $\dim Q = n$ and $Q$ has a basis $\{ q_1, \ldots, q_n \}$ over $k$. Consider the Lie algebra
$$
L = A \leftthreetimes Q
$$ where $A = k[x]$ and $[q_i, a] = x^i a$ for $a \in A$. Then 

a) $L$ is a faithful contracting recurrent self-similar metabelian Lie algebra with a generating set, where each element is finite state;

b) if $char(k) = p > 0$ then $L$ is transitive self-similar Lie algebra; 

c) $L$ is not regularly weakly branched;

d) if $char(k) = p > n$ or $char(k) = 0$ then $L$ is of homological type $FP_n$.}

\medskip The proof of Theorem D is based on the construction of virtual endomorphism $\theta : H \to L$, where $H = A_0 \leftthreetimes Q$ and $A_0 = x k[x]$, $\theta$ is the identity on $Q$ and for $a \in A_0$ we have $\theta(a) = a/ x$.
 The definitions of transitive, of  regularly weakly branched and of contracting recurrent self-similar Lie algebra can be found in Section \ref{def1234}. We note that when $n = 1$ and $char(k) =
p > 0$, the Lie algebra $L$ in the above theorem is the Lie algebra variant of the classical lamplighter group.

Moreover, we note that in Theorem D the $U(Q)$-module $A$  has Krull dimension 1. The same condition on the Krull dimension appears in a recent result for metabelian self-similar groups in \cite{DesiSaid} : every finitely generated group $G = B \rtimes Q_0$ with $B$ and $Q_0$ abelian, $Q_0$ torsion-free and $B$ of Krull dimension 1 as $\mathbb{Z} Q_0$-module is self-similar.

  \section{Preliminaries} \label{prel}
  
  \subsection{The structure of self-similar Lie algebra}   \label{section-structure}
  By definition
  $$
   L \wr {\mathcal Der } X = (X \otimes_{k} L) \leftthreetimes {\mathcal  Der} X
   $$
   is a Lie algebra with the following  operations :
   $$
   [x_1 \otimes a_1, x_2 \otimes a_2] = (x_1 x_2) \otimes [a_1, a_2] \hbox{ for } x_1, x_2 \in X, a_1, a_2 \in L
   $$
   and
   $$
   [\delta, x \otimes a] = \delta(x) \otimes a \hbox{ for } x \in  X, a \in L, \delta \in {\mathcal Der} (X)
   $$ 
  In \cite{Bartholdi} Bartholdi defined  inductively   an action of $L$ on $ X^{\otimes m}$ for $m \geq 1$. Suppose $$ \psi (a) = \sum_i z_i \otimes a_i + \delta_a, \delta_a \in {\mathcal Der}X.$$  Then the action of $a \in L$ on $x \in X$ is the following one :
  $$
   a \cdot x = \delta_a(x)
   $$
   and
   $$ 
   a \cdot (x_1 \otimes \ldots \otimes x_m) = \sum_i (z_i x_1) \otimes {a_i}  \cdot (x_2 \otimes \ldots \otimes x_m) + \delta_a(x_1) \otimes x_2 \otimes \ldots \otimes x_m,
   $$
   where $z_i x_1$ is the product of $z_i$ and $x_1$ in $X$.
   Thus there is a Lie algebra homomorphism $L \to End_k (X^{\otimes m})$ that is the homomorphism $\nu$ from the Section \ref{intr}.
   
   If we consider $T(X)$ as a ring with product that equals $\otimes$ then the action of $L$ on $T(X)$ constructed above is not by derivations. But if we consider $X^{\otimes m}$ as a ring via $(x_1 \otimes \ldots \otimes x_m)(y_1 \otimes \ldots \otimes y_m) = (x_1 y_1) \otimes \ldots \otimes (x_m y_m)$ then the action of $L$ on $X^{\otimes m}$ descibed above is by derivations. This can be easily seen by induction on $m$. Observe that $L$ is a $\psi$-self-similar   faithful Lie algebra if
   $$ \cap_{m \geq 1} Ker (L \to  End_k (X^{\otimes m})) = 0.$$
 
 \subsection{Some definitions}  \label{def1234}
 
 We recall from \cite{Bartholdi} that $L$ is a recurrent $\psi$-self-similar Lie algebra if  the corresponding virtual endomorphism $\theta : H \to L$ is surjective.

Furthermore to define contracting $\psi$-self-similar Lie algebra we need the notion of states. Let $L$ be a $\psi$-self-similar Lie algebra. Following the definition of finite state element of $L$ from \cite{Bartholdi}  we define for a subset $Y$ of $L$  the spaces of states $S(Y)$ of $Y$ as $S(Y) = V$, where $V$ is the smallest (under inclusion) subspace of $L$ such that $\psi(Y) \subseteq (X \otimes V)  \leftthreetimes {\mathcal Der}X$. 
Then $S^i(Y) = S(S^{i-1}(Y))$ for $i \geq 2$. Then following \cite{Bartholdi} we say that $L$ is a contracting $\psi$-self-similar Lie algebra if there is a finite dimensional subspace $S$ of $L$  such that for every $a \in L$ there is an integer $m_0 \geq 1$ depending on $a$  such that for every $m \geq m_0$ we have that $S^m(a) \subseteq S$.

By \cite{Bartholdi} a recurrent transitive self-similar Lie algebra $L$ is regularly weakly branched if there is a non-zero ideal $K$ of $L$ such that $\psi(K)$ contains $X \otimes K$. If furthermore $K$ has finite codimension in $L$ we say that $L$ is regularly branched.
   \subsection{Constructing a virtual endomorphism for self-similar Lie algebras} \label{def-endo}
   
   The description of self-similar groups can be done by virtual endomorphism. In the case of actions on the rooted binary tree ${\mathcal T}_2$ this was first considered by   Nekrashevich and Sidki in \cite{Nekra}, \cite{N-S}. For every  group $G$   with state close, transitive representation in the automorphism group of a one rooted homogeneous tree ${\mathcal T_m}$, where every vertex has $m$ descendents,  there is a subgroup of finite index $\widetilde{G}$ together with a virtual endomorphism \begin{equation} \label{vir} f : \widetilde{G} \to G \end{equation}  constructed from the action of $G$ on the tree ${\mathcal T}$ i.e. $f$ is a homomorphism of groups. The converse holds too, for every virtual endomorphism (\ref{vir}) it is possible to construct a state closed action of $G$ on a one rooted homogeneous  tree ${\mathcal T}_m$, where $m = [G : \widetilde{G}]$, such that the action is transitive on the first level of the tree and this action is faithful (i.e. with a trivial kernel) precisely when $f$ is simple (i.e. there isn't a non-trivial, normal subgroup $K$ of $G$  contained in $H$ such that $f(K) \subseteq K$). Here we suggest a Lie algebra analogy of virtual endomorphism and use this to construct  self-similar Lie algebras.
   
    Suppose that $L$ is a self-similar Lie algebra i.e. satisfies condition (\ref{def01}). 
   Let $$\pi :   L \wr {\mathcal Der } X \to   {\mathcal Der } X$$ be the canonical projection. Define
   $$
   H = Ker (\pi \psi).
   $$
   Thus $H$ is an ideal of $L$.
    We assume from now on that that the commutative ring $X$ is endowed with  a surjective  homomorphism of $k$-algebras
    $$
    \epsilon : X \to k.
    $$
     Then we define a homomorphism of Lie algebras
     $$
     \theta : H \to L
     \ \ \hbox{ by }
 \ \      \theta = (\epsilon \otimes id_L) \psi
     $$
     i.e. $\theta(h) = \sum_i \epsilon (x_i) a_i$ when $\psi (h) = \sum_i x_i \otimes a_i$. In analogy with the group case we call $\theta$ the virtual endomorphism  associated to $\psi$ though here there is no analogy of the notion of subgroup of finite index. If $X$ is finite dimensional over $k$, then ${\mathcal Der} X$ is finite dimensional (over $k$), so $H$ has finite codimension in $L$. This justifies the word virtual in the definition of virtual endomorphism $\theta$.

   \subsection{The universal $(X,L_0)$-faithful self-similar Lie algebra} \label{universal}
In this section we assume that $X$ is {\bf finite dimensional} over $k$.   
   Define for a Lie subalgebra $L_0$ of ${\mathcal Der}X$ the vector space
   $${\mathcal W}(X, L_0) = \prod_{n = 0}^{\infty} (X^{\otimes n} \otimes L_0).
   $$
   The case when $L_0 = {\mathcal Der}X$ was considered in \cite{Bartholdi} and denote by ${\mathcal W}(X)$. In \cite{Bartholdi}  ${\mathcal W}(X)$ is endowed with Lie algebra structure and if $char(k) = p > 0$ then  ${\mathcal W}(X)$ is a restricted Lie algebra.  
   If $a = x_1 \otimes \ldots \otimes x_m \otimes \delta \in X^{\otimes m} \otimes {\mathcal Der}X$ and $b = y_1 \otimes \ldots \otimes y_n \otimes \epsilon \in X^{\otimes n} \otimes {\mathcal Der}X$ the bracket $[a,b]$ was defined in \cite{Bartholdi} by
   $$[a,b] = x_1 y_1 \otimes \ldots \otimes x_m y_n \otimes [\delta, \epsilon] \hbox{ if } n = m,$$
   $$ 
   [a,b] =  x_1 y_1 \otimes \ldots \otimes x_m y_m \otimes \delta(y_{m+1}) \otimes \ldots \otimes y_n \otimes \epsilon
   \hbox{ if } m < n$$
   and
   $$[a,b] = - [b,a] =  - x_1 y_1 \otimes \ldots \otimes  x_m y_m \otimes \epsilon(x_{n+1}) \otimes \ldots \otimes x_m \otimes \delta \hbox{ if } n < m.
   $$
   Note that ${\mathcal W}(X, L_0)$ is a Lie subalgebra of ${\mathcal W}(X)$. The self-similar structure of ${\mathcal W}(X, L_0)$ as defined in \cite{Bartholdi} is
   $$
   \psi_0 : {\mathcal W}(X, L_0) \to  (X \otimes  {\mathcal W}(X, L_0)) \leftthreetimes L_0
   $$
   where
   $$
   \psi_0 (a_0, a_1, \ldots, a_n, \ldots ) = (a_1, a_2, \ldots, a_n,\ldots ) + a_0
   $$
and we have identified $ \prod_{n = 1}^{\infty} (X^{\otimes n} \otimes L_0)$ with  $ X \otimes \prod_{n = 0}^{\infty} (X^{\otimes n} \otimes L_0)$. The last requires that $X \otimes - $ is a functor that commutes with direct product and this is the case when $X$ is finite dimensional over $k$.

In \cite{Bartholdi} is observed (without a proof) that every faithful self-similar Lie algebra for a fixed $X$  (called alphabet  in \cite{Bartholdi}) embeds in ${\mathcal W} (X)$. A similar fact holds for faithful $(X,L_0)$-self-similar Lie algebras and the universal $(X,L_0)$-self-similar Lie algebra ${\mathcal W}(X, L_0)$.
 
 \begin{lemma} Let $X$ be a { finite dimensional} commutative $k$-algebra.  Then if $L$ is a faithful $\psi$-self-similar Lie algebra with $$
   \psi : L \to  (X \otimes L)  \leftthreetimes L_0
   $$
   for some Lie subalgebra $L_0$ of ${\mathcal Der} X$
 then $L$ embeds in ${\mathcal W}(X, L_0)$.
 \end{lemma}
 
  \begin{proof} We define inductively homomorphisms of Lie algebras
  $$
  \delta_i : L \to M_i = \prod_{0 \leq j \leq i} (X^{\otimes j} \otimes L_0)
  $$ 
  by $\delta_0 = \pi \psi$, where $\pi : (X \otimes_{k} L)  \leftthreetimes L_0 \to L_0$ is the canonical projection. To define $\delta_{i+1}$ define first 
  $$\mu_{i+1} : (X \otimes L)  \leftthreetimes L_0 \to
  X \otimes (\prod_{0 \leq j \leq i} (X^{\otimes  j} \otimes L_0)) \leftthreetimes L_0
  $$
  by $\mu_{i+1} = (id_X \otimes \delta_i) \leftthreetimes id_{L_0}$ and then set
  $$
  \delta_{i+1} = \mu_{i+1} \psi.
  $$
  Write $H = Ker (\delta_0) = Ker (\pi \psi)$. Then using that $\psi$ is injective we get that
  $$Ker (\delta_{i+1}) = \{ a \in L \ | \ \psi(a) \in X \otimes Ker (\delta_i) \}.$$
  By induction on $i$ it follows that $Ker (\delta_i) = Ker (L \to End_k( X^{\otimes i}))$, hence by the faithfulness of $\psi$ we get that $\cap_i Ker (\delta_i) = 0$.
  
  Let $\pi_{i+1} : M_{i+1} \to M_i$ be the canonical projection. Using the definition of the map $\delta_i$ it is easy to check that 
  $\pi_{i+1} \delta_{i+1} = \delta_i$ for every $i \geq 0$. Thus the maps $\{\delta_i \}_{i \geq 0}$ induce a homomorphism of Lie algebra $\delta : L \to  {\mathcal W}(X, L_0)$, whose kernel is $\cap_i Ker (\delta_i) = 0$. 
 
\end{proof}
    \section{The proof of Theorem A}
   
    We use the following notation : for $b,a \in L$ we define $b \circ a = [b,a]$ and for $\lambda$ an element of the universal enveloping algebra  $ U(L)$ the element $\lambda \circ a$ is obtained from $a$ by acting (on the left) with $\lambda$ i.e. for $\lambda = \lambda_1 \lambda_2$ we have $(\lambda_1 \lambda_2) \circ a = \lambda_1 \circ (\lambda_2 \circ a)$. In particular for  $m \in {\mathbb N}$ we have $b^m \circ a = b^{m-1} \circ [b,a]$.

a)     Let $Y = \{ x_1^{z_1} \ldots x_n^{z_n} \}_{z_1 \geq  0, \ldots, z_n \geq 0}$ be a basis of $X$ as a vector space over $k$ if $char(k) = 0$ and $Y  = \{ x_1^{z_1} \ldots x_n^{z_n}  \}_{ 0 \leq z_1, \ldots, z_n \leq p-1}$ if $char(k) = p$.

      Suppose that
     $$
     \psi(h) = \sum_{0 \leq i_1, \ldots, i_n \leq m_0} x_1^{i_1} \ldots x_n^{i_n}  \otimes a_{i_1, \ldots, i_n}\hbox{ for } h \in H \hbox{ and } \psi(y_i) = \partial / \partial_{x_i}.
     $$
     Then for $m_0 = m-1$, where $m$ is given  by condition 3 from the statement of Theorem A,
     $$
      \psi (y_1^{z_1} \ldots y_n^{z_n} \circ h)  =
      \psi(y_1^{z_1} \ldots y_n^{z_n}) \circ \psi(h) = $$ $$
      (\partial / \partial x_1)^{z_1} \ldots (\partial/ \partial x_n)^{z_n} \circ \sum_{0 \leq i_1, \ldots, i_n \leq m_0} x_1^{i_1} \ldots x_n^{i_n}  \otimes a_{i_1, \ldots, i_n} = $$
      $$ \sum_{0 \leq i_1, \ldots, i_n \leq m_0}
      (\partial / \partial x_1)^{z_1} \ldots (\partial/ \partial x_n)^{z_n}   (x_1^{i_1} \ldots x_n^{i_n})  \otimes a_{i_1, \ldots, i_n} = 
      $$
      $$
      \sum_{0 \leq i_1, \ldots, i_n \leq m_0}
      (\partial / \partial x_1)^{z_1} (x_1^{i_1} ) \ldots (\partial/ \partial x_n)^{z_n}   (x_n^{i_n})  \otimes a_{i_1, \ldots, i_n} =
      $$
      $$\sum_{0 \leq i_1, \ldots, i_n \leq m_0}
      i_1 \ldots (i_1- z_1+1) x_1^{i_1 - z_1} \ldots  (i_n \ldots (i_n- z_n + 1)) x_n^{i_n - z_n}  \otimes a_{i_1, \ldots, i_n}
      $$
      Then
       \begin{equation} \label{reforma}
      \theta(y_1^{z_1} \ldots y_n^{z_n} \circ h) = 
   (\epsilon \otimes id) (\psi (y_1^{z_1} \ldots y_n^{z_n} \circ h) ) =   \end{equation}
   $$
   \sum_{0 \leq i_1, \ldots, i_n \leq m_0}
      i_1 \ldots (i_1- z_1+1) \epsilon( x_1)^{i_1 - z_1} \ldots  (i_n \ldots (i_n- z_n + 1)) \epsilon(x_n)^{i_n - z_n}   \otimes a_{i_1, \ldots, i_n}=
   $$
  $$
   z_1 ! \ldots z_n ! a_{z_1, \ldots, z_n}, $$
   where the last equality follows from $\epsilon(x_i) = 0$ for every $1 \leq i \leq n$.
   Thus
   $$a_{z_1, \ldots, z_n} = ( z_1 ! \ldots z_n !)^{-1}  \theta(y_1^{z_1} \ldots y_n^{z_n} \circ h) \hbox{ if } char(k) = 0 \hbox{ or each } z_i \leq char(k) - 1  
    $$     
     Note that  by condition 3 from the statement of Theorem A
     $$\theta(y_1^{z_1} \ldots y_n^{z_n}  \circ h) = 0 \hbox{ if some } z_i \geq m_0 + 1 = m.
     $$
     Thus (\ref{main-eq}) holds.

  b)   Note that by  (\ref{main-eq})
     $$
     Ker (\psi) = \{ h \in H \mid \theta(y_1^{i_1} \ldots y_n^{i_n} \circ h) = 0 \hbox{ for all } i_1, \ldots, i_n \geq 0 \}.
     $$
     Observe that $Ker (\psi)$ is an ideal of $H$, since $\psi$ is a homomorphism of Lie algebras. Furthermore $$[y_j, Ker (\psi)] =  y_j \circ Ker (\psi) \subseteq Ker (\psi) \hbox{ for } 1 \leq j \leq n,$$
     hence $Ker ( \psi)$ is an ideal of $L$ contained in $Ker(\theta)$. 
     
c) Consider when $L$ is faithful $\psi$-self-similar Lie algebra.
  By (\ref{reforma}) and by the definition of the $L$ action  on $ X^{\otimes m}$ we deduce that for $v_1, \ldots, v_m \in X$
\begin{equation} \label{action-new}
h \cdot (v_1 \otimes \ldots \otimes v_m) =  \end{equation} $$ \sum_{i_1, \ldots, i_n} (i_1!)^{-1} \ldots (i_n!)^{-1} (x_1^{i_1} \ldots x_n^{i_n} v_1) \otimes \theta(y_1^{i_1} \ldots y_n^{i_n} \circ h)  \cdot (v_2 \otimes \ldots \otimes v_m).
$$
Thus $h \in H$ acts trivially on $ X^{\otimes m} $ if and only if $\theta(y_1^{i_1} \ldots y_n^{i_n} \circ h)$ acts trivially on $ X^{\otimes {(m-1)}}$ for every $i_1, \ldots, i_n \geq 0$. Note that $a \in L$ acts trivially on $X$ if and only if $a \in H$.

Let $I = \cap_m Ker (L \to End_k ( X^{\otimes m}))$. Then $I$ is an intersection of ideals of $L$, so is an ideal itself and $I \subseteq Ker (L \to {\mathcal Der } X) = H$.  By the above remark for every $h \in I$ we have that $\theta(y_1^{i_1} \ldots y_n^{i_n} \circ h) \in I$, in particular 
for $i_1 = \ldots = i_n = 0$ we get that $\theta(h) \in I$. Thus $I$ is an $\theta$-invariant ideal of $L$ that is contained in $H$. If any such ideal is trivial then $I$ is trivial and $L$ is faithful $\psi$-self-similar.

If $J$ is a non-zero $\theta$-invariant ideal of $L$ ( i.e. $\theta(J) \subseteq J$) and $J \subseteq H$ then for every $h \in J$ we have that $
y_1^{i_1} \ldots y_n^{i_n} \circ h \in J$ and $\theta(y_1^{i_1} \ldots y_n^{i_n} \circ h) \in \theta(J) \subseteq J$. We prove by induction on $i$ that $J$ acts trivially on $X^{\otimes i} $. Indeed since $J \subseteq H$ we get that $J$ acts trivially on $X$. For the inductive step assume that $J$ acts trivially on $ X^{ \otimes {(i-1)} }$ and note that $\theta(y_1^{i_1} \ldots y_n^{i_n} \circ J) \subseteq J$, so $\theta(y_1^{i_1} \ldots y_n^{i_n} \circ J)$ acts trivially on   $ X^{\otimes {(i-1)}}$, hence $J$ acts trivially on  $ X^{ \otimes {i} }$ by ( \ref{action-new} ).

d) By (\ref{main-eq}) for a vector subspace $W$ of $H$ we have $\psi(W) \subseteq X \otimes W$ if and only if for every $w \in W$ we have that 
$\theta( y_1^{i_1} \ldots y_n^{i_n} \circ w) \in W$ for all $ 0 \leq i_1, \ldots, i_n \leq m_0 = m-1$.

e) Finally we prove by induction on $m$ that if $\theta(H) = L$ then $U(L) \cdot (x^{p-1} \otimes \ldots \otimes x^{p-1}) =  X^{ \otimes m} $ i.e.  $L$ is a transitive $\psi$-self-similar Lie algebra. By the description of the $L$ action on $X^{ \otimes m } $ and  (\ref{main-eq})  we get that if $X = k[x]/ (x^p)$, $char(k) = p >0$ then for $v_m = x^{p-1} \otimes \ldots \otimes x^{p-1} \in  X^{ \otimes m }$ and $h_1, \ldots, h_s \in H$
$$
h_1 \ldots h_s \cdot  v_m= x^{p-1} \otimes \theta(h_1) \ldots \theta(h_s) \cdot v_{m-1}
$$
Indeed let $h \in H$ and $\psi(h) = \sum_i \widetilde{y}_i \otimes b_i$ for some $\widetilde{y}_i \in X, b_i \in L$. Then $\theta(h) = \sum_i \epsilon(\widetilde{y}_i) b_i$. Since $x^p = 0$ we have that $\widetilde{y}_i x^{p-1} = \epsilon(\widetilde{y}_i) x^{p-1} \in k x^{p-1}$ and so
$$
h \cdot v_m = \sum_i (\widetilde{y}_i x^{p-1}) \otimes (b_i \cdot v_{m-1}) = 
 \sum_i (\epsilon(\widetilde{y}_i) x^{p-1}) \otimes (b_i \cdot v_{m-1}) =$$ $$
 x^{p-1} \otimes  ( \sum_i \epsilon(\widetilde{y}_i) b_i) \cdot v_{m-1}) =
 x^{p-1} \otimes (\theta(h) \cdot v_{m-1}).
 $$ 
Then for $z \geq 1$  and $\psi(y_1) = \partial / \partial x \in {\mathcal Der} X$
$$
y_1^z \cdot (x^{p-1} \otimes \theta(h_1) \ldots \theta(h_s) \cdot v_{m-1}) = (\partial / \partial x_1)^z (x^{p-1})  \otimes (\theta(h_1) \ldots \theta(h_s) \cdot v_{m-1}) =
$$
$$
(p-1) \ldots (p-z) x^{p-1-z} \otimes \theta(h_1) \ldots \theta(h_s) \cdot v_{m-1}$$
By induction $U(L) \cdot v_{m-1} =  X^{ \otimes {(m-1)} }$ and $\theta(H) = L$, so
$ \sum_{s \geq 1; h_1, \ldots, h_s \in H} \theta(h_1) \ldots \theta(h_s) \cdot v_{m-1} =  X^{ \otimes {(m-1)} }$. This implies that $U(L) \cdot v_m = X^{  \otimes m }$.

         \section{Constructing self-similar Lie algebras  from virtual endomorphism : the case when the derivation part is abelian} \label{construct}
     
     In Section \ref{def-endo} we constructed a virtual endomorphism $\theta$ associated to a self-similar Lie algebra $L$. In this section 
     we start with a virtual endomorphsm of Lie algebras $\theta : H \to L$ and study   necessary conditions for $L$ to be $(X,L_0)$-self-similar Lie algebra  with associated virtual endomorphism $\theta$, when $L_0$ is abelian.

    \medskip
  {\bf Proof of Theorem B}

By Theorem A we do not have a choice but  define for $h \in H$
\begin{equation} \label{reforma2}
\psi(h) =
 \sum_{0 \leq i_1, \ldots, i_n \leq m_0} ( i_1 ! \ldots i_n !)^{-1} x_1^{i_1} \ldots x_n^{i_n}  \otimes 
   \theta(y_1^{i_1} \ldots y_n^{i_n} \circ h)
\end{equation}
where if $char(k) = p$ then $m_0 = p-1$ and if $char(k) = 0$ then $m_0 = m-1$ from the statement of the theorem.

\medskip
We claim that $\psi$ is a homomorphism of Lie algebras. Note that $\psi$ is a linear map by construction.
Let $h_1, h_2 \in H$. Then
$$
[\psi(h_1), \psi(h_2)] =
[
 \sum_{0 \leq i_1, \ldots, i_n \leq m_0}  ( i_1 ! \ldots i_n !)^{-1}  x_1^{i_1} \ldots x_n^{i_n}  \otimes 
 \theta(y_1^{i_1} \ldots y_n^{i_n} \circ \  h_1),$$ $$
 \sum_{0 \leq j_1, \ldots, j_n \leq m_0}  ( j_1 ! \ldots j_n !)^{-1} x_1^{j_1} \ldots x_n^{j_n}  \otimes 
  \theta(y_1^{j_1} \ldots y_n^{j_n} \circ \ h_2)
]=
$$
$$
\sum_{0 \leq i_1, \ldots, i_n, j_1, \ldots, j_n \leq m_0} ( i_1 ! \ldots i_n !)^{-1}  ( j_1 ! \ldots j_n !)^{-1} x_1^{i_1+ j_1} \ldots x_n^{i_n + j_n}
  \otimes $$ $$[\theta(y_1^{i_1} \ldots y_n^{i_n} \circ \  h_1), \theta(y_1^{j_1} \ldots y_n^{j_n} \circ \ h_2)].
  $$
Observe that for an arbitrary element $b \in L$ we have
 $$
 b^k \circ [h_1, h_2] = \sum_{i + j = k} {k \choose i} [b^i \circ h_1, b^j \circ h_2] 
 $$
and hence
$$
(y_1^{k_1} \ldots y_n^{k_n}) \circ [h_1, h_2] = \sum_{i_s + j_s = k_s} ( \prod_{1 \leq s \leq n} {{k_s} \choose {i_s}}) [ (y_1^{i_1} \ldots y_n^{i_n}) \circ h_1, (y_1^{j_1} \ldots y_n^{j_n}) \circ h_2]   
$$
Then for $h_1, h_2 \in H$ we have
$$
\psi([h_1, h_2]) = 
 \sum_{0 \leq k_1, \ldots, k_n \leq m_0} ( k_1 ! \ldots k_n !)^{-1} x_1^{k_1} \ldots x_n^{k_n}  \otimes 
   \theta(y_1^{k_1} \ldots y_n^{k_n} \circ [h_1, h_2]) =
$$
 $$ \sum_{0 \leq k_1, \ldots, k_n \leq m_0}
 \sum_{i_s + j_s = k_s} ( k_1 ! \ldots k_n !)^{-1} (\prod_{1 \leq s \leq n} {{k_s} \choose {i_s}}) x_1^{k_1} \ldots x_n^{k_n}  \otimes $$ $$[ \theta((y_1^{i_1} \ldots y_n^{i_n}) \circ h_1), \theta( (y_1^{j_1} \ldots y_n^{j_n}) \circ h_2)] =   
 $$
 $$
 \sum_{0 \leq i_1, \ldots, i_n, j_1, \ldots, j_n \leq m_0} ( i_1 ! \ldots i_n !)^{-1}  ( j_1 ! \ldots j_n !)^{-1} x_1^{i_1+ j_1} \ldots x_n^{i_n + j_n} \otimes$$ $$
  [\theta(y_1^{i_1} \ldots y_n^{i_n} \circ \  h_1), \theta(y_1^{j_1} \ldots y_n^{j_n} \circ \ h_2)] = [\psi(h_1), \psi(h_2)].
 $$
 We have defined $\psi$ on $H$ by (\ref{reforma2}), the restriction of $\psi$ on $L_0$ is defined by condition 2 from the statement of Theorem B. Thus $\psi$ is defined on $H \rtimes L_0$ and
  to complete the proof of the fact that $\psi$ is a Lie algebra homomorphism we need to show that
 $$
 [\psi(y_j), \psi(h)] = \psi(y_j \circ h) = \psi([y_j,h]) \hbox{ for } h \in H, 1 \leq j \leq n.
 $$
 Note that
 $$
 [\psi(y_j), \psi(h)] =[\partial/ \partial x_j, 
 \sum_{0 \leq i_1, \ldots, i_n \leq m_0} ( i_1 ! \ldots i_n !)^{-1} x_1^{i_1} \ldots x_n^{i_n}  \otimes 
   \theta(y_1^{i_1} \ldots y_n^{i_n} \circ h) ] = 
 $$
 $$\sum_{0 \leq i_1, \ldots, i_n \leq m_0} ( i_1 ! \ldots i_n !)^{-1} \partial/ \partial x_j(x_1^{i_1} \ldots x_n^{i_n})  \otimes 
   \theta(y_1^{i_1} \ldots y_n^{i_n} \circ h) =
  $$
  $$
 \sum_{0 \leq i_1, \ldots, i_n \leq m_0} ( i_1 ! \ldots (i_j - 1)! \ldots  i_n !)^{-1} x_1^{i_1} \ldots x_j^{i_j - 1} \ldots  x_n^{i_n}  \otimes 
   \theta(y_1^{i_1} \ldots y_n^{i_n} \circ h).$$
  On other hand
  $$
  \psi([y_j,h])=  \sum_{0 \leq i_1, \ldots, i_n \leq m_0} ( i_1 ! \ldots i_n !)^{-1} x_1^{i_1} \ldots x_n^{i_n}  \otimes 
   \theta(y_1^{i_1} \ldots y_n^{i_n} \circ [y_j,h]) =
  $$
  $$
  \sum_{0 \leq i_1, \ldots, i_n \leq m_0} ( i_1 ! \ldots i_n !)^{-1} x_1^{i_1} \ldots x_n^{i_n}  \otimes 
   \theta(y_1^{i_1} \ldots y_j^{i_j + 1} \ldots y_n^{i_n} \circ h) =
   $$
  $$
   \sum_{0 \leq i_1, \ldots, i_{j-1}, i_{j+1}, i_n \leq m_0} ( i_1 ! \ldots i_{j-1}! m_0! i_{j+1}! \ldots i_n !)^{-1} x_1^{i_1} \ldots x_{j-1}^{i_{j-1}} x_j^{m_0}  x_{j+1}^{i_{j+1}} \ldots  x_n^{i_n} $$ $$ \otimes 
   \theta(y_1^{i_1} \ldots y_j^{m_0 + 1} \ldots y_n^{i_n} \circ h))
  + [\psi(y_j), \psi(h)] = [\psi(y_j), \psi(h)].$$
  The last equality follows since
  $$
  y_1^{i_1} \ldots y_j^{m_0 + 1} \ldots y_n^{i_n} \circ h = 
   y_j^{m_0 + 1} \circ (y_1^{i_1} \ldots y_{j-1}^{i_{j-1}} y_{j+1}^{i_{j+1}} \ldots y_n^{i_n} \circ h)
   $$
   and so
  $$
  \theta(y_1^{i_1} \ldots y_j^{m_0 + 1} \ldots y_n^{i_n} \circ h) = 0.$$

     \section{Constructing self-similarity from virtual endomorphism : the  case when the derivation part is  either a subalgebra of $sl_2$ or  the Witt algebra}    \label{sl2-Witt}
   
     Recall that for 
      $X = k[x]/ (x^p)$, where $k$ is a field of characteristic $p > 0$, the Lie algebra ${\mathcal Der}X$ is called the Witt algebra or Jacobson-Witt algebra.
      It has a basis as a vector space over $k$
      $$
      \{ e_i = x^{i+1} \partial / \partial x \}_{-1 \leq i \leq p-2}
      $$
      and has the following relations
      \begin{equation} \label{relations}
      [e_i, e_j] = (j-i) e_{i+j}.
      \end{equation}
      Observe that $[e_{-1},e_0 ] = e_{-1}, [e_{-1}, e_1] = 2 e_0$ and $[e_0, e_1] = e_1$. Thus the span of $e_{-1}, e_0$ and $e_1$ is a Lie    subalgebra of ${\mathcal Der}X$ isomorphic to $sl_2(k)$ if $char(k) \not= 2$.
      
      The following property of the Witt algebra ${\mathcal Der}(k[x]/(x^p))$ is probably known but as we could not find a reference we give a proof.
      
      \begin{lemma} Let $L_0$ be a proper subalgebra of the Witt algebra 
      ${\mathcal Der}(k[x]/(x^p))$, where $char(k) = p > 0$, such that $L_0$   contains $e_{-1} = \partial/ \partial x$. Then $L_0$ is contained in the span of $e_{-1}, e_0, e_1$.
      \end{lemma}
      
      \begin{proof} Let $f \in L_0 \setminus k e_{-1}$, then for some $k_{-1}, k_0, \ldots, k_{p-2} \in k$ we have
      $
      f =k_{-1} e_{-1} +  k_0 e_0 + k_1 e_1 + \ldots + k_{p-2} e_{p-2}
      $. Suppose that $k_{j_0} \not= 0$ and $k_j = 0$ for all $j > j_0$ and for all possible $f$. If $j_0 \leq 1$ we are done. From now on we can suppose that $2 \leq j_0$.
      Then $$f_1 =f - k_{-1} e_{-1} =  k_0 e_0 + k_1 e_1 + \ldots + k_{j_0} e_{j_0} \in L_0.$$ 
      Observe that 
      $$
      e_{-1}^{j_0} \circ f_1 = k_{j_0 - 1} j_0! e_{-1} + k_{j_0} (j_0+1)! e_0 \in L_0 \hbox{ and so } k_{j_0} (j_0+1)! e_0 \in L_0. $$
      Since $k_{j_0} \not=0$ we deduce that $$e_0 \in L_0.$$ Then
      $$
      e_{-1}^{j_0-1} \circ f_1 \in k e_{-1} \oplus k e_0 \oplus k_{j_0} (j_0 + 1) \ldots 3 e_1
      $$
      is an element of $L_0$
      and since $e_0, e_{-1} \in L_0$ we get that
      $$e_1 \in L_0.$$
      Considering $
      e_{-1}^{j_0-2} \circ f_1$ and repeating  the above calculation we deduce that $e_2 \in L_0$. Continuing in the same fashion we get that  $$e_0, \ldots, e_{j_0} \in L_0.$$   Since $L_0$ is a proper Lie subalgebra of the Witt algebra we have  that $j_0 \leq p-3$. Then $[e_1, e_{j_0}] = (j_0 - 1) e_{1 + j_0} \in L_0$ a contradiction with the choice of $j_0$.
      \end{proof}
     \begin{theorem} \label{sl-2}
     Let $X = k[x]/ (x^p)$ where $k$ is a field of positive characteristic $p \not= 2$.
     Let $L$ be a Lie algebra over $k$, $H$ an ideal of $L$ such that : 
     
      a)  the short exact sequence of Lie algebras $H \to L \to L/H$ splits, i.e. $$L = H  \leftthreetimes L_0, \hbox{ where }L_0 \simeq L/H$$
      
         b) $L_0 $ is a Lie subalgebra of the Witt algebra ${\mathcal Der}X$ that contains $e_{-1}$ i.e.  $L_0$ has a basis $\{ e_{-1}, \ldots, e_{j_0} \}$ as a vector space over $k$, where either $j_0 = p - 2$   or $j_0 \leq 1$.   
         
         Let $$\theta : H \to L$$ be a Lie algebra homomorphism such that for every $h \in H$ :
         
         1.  $\theta(e_{-1}^p \circ h) = 0$;
         
         2. $
   \theta(e_{-1}^i e_j \circ h) = 0 \hbox{ for } 0 \leq i \leq j, 0 \leq j \leq j_0$;
   
         3. $
   i !  
   \theta(e_{-1}^{i-j} \circ h) = ( i-j-1) !   
   \theta(e_{-1}^i e_j \circ h) \hbox{ for } 0 \leq j \leq j_0,  j+1 \leq i \leq p-1.$

     Then
      there is a homomorphism of Lie algebras
      $$\psi :  L \to L \wr {\mathcal Der } X = (X \otimes L) \leftthreetimes {\mathcal Der}X
     $$
     such that 
     $\theta$ is the virtual endomorphism associated to $\psi$ with respect to the augmentation map $\epsilon : X \to k$ (i.e. $\epsilon(x) = 0$) and 
    the restriction of $\psi$ on $L_0$ is the inclusion of $L_0$ in ${\mathcal Der}X$. 
     \end{theorem}
     {\bf Remark} If $p = 2$ there are more options for $L_0 \subseteq k e_{-1} \oplus k e_0 \oplus k e_1$.
     \begin{proof}
     By the proof of Theorem B applied for the Lie algebra $H \leftthreetimes k e_{-1}$ i.e. $n = 1$, $x = x_1$ and $y_1 = e_{-1} = \partial/ \partial x_1$, we deduce there is a homomorphism of Lie algebras
     $$
     \psi : H \to X \otimes L
     $$
     given by
     $$
\psi(h) =
 \sum_{0 \leq i \leq p-1} ( i ! )^{-1} x^{i} \otimes 
   \theta(e_{-1}^i \circ h)
$$
and 
$$
      \psi([e_{-1}, h]) = [\psi(e_{-1}), \psi(h)] \hbox{ for every } h \in H.
     $$
     Now we define
     $$\psi : L_0 \to (X \otimes L) \leftthreetimes {\mathcal Der} X
     $$
     as the inclusion of $L_0$ in ${\mathcal Der}X$.

     Finally to prove that this gives a homomorphism of Lie algebras
     $$
     \psi : L \to (X \otimes L) \leftthreetimes {\mathcal Der} X
     $$
     we have to show that 
     \begin{equation} \label{condition1}
      \psi([e_{j}, h]) = [\psi(e_{j}), \psi(h)] \hbox{ for }0 \leq j \leq j_0 \hbox{ and } h \in H.
     \end{equation}
     Note that
     $$
     [\psi(e_j), \psi(h)] = [e_j, 
 \sum_{0 \leq i \leq p-1} ( i ! )^{-1} x^{i} \otimes 
   \theta(e_{-1}^i \circ h)] = 
 \sum_{0 \leq i \leq p-1} ( i ! )^{-1} e_j(x^{i}) \otimes 
   \theta(e_{-1}^i \circ h) =$$
   $$
   \sum_{1 \leq i \leq p-1} ( (i-1) ! )^{-1} x^{i+j} \otimes 
   \theta(e_{-1}^i \circ h) = 
     \sum_{j+1 \leq i \leq p-1} ( (i-j-1) ! )^{-1} x^{i} \otimes 
   \theta(e_{-1}^{i-j} \circ h). 
   $$
   On other hand
   $$
   \psi([e_j,h]) =  \psi(e_j \circ h ) =
 \sum_{0 \leq i \leq p-1} ( i ! )^{-1} x^{i} \otimes 
   \theta(e_{-1}^i e_j \circ h).
   $$
   Then (\ref{condition1}) is equivalent to
   \begin{equation} \label{condition2}
   \theta(e_{-1}^i e_j \circ h) = 0 \hbox{ for } 0 \leq i \leq j, 0 \leq j \leq j_0 
   \end{equation}
   and
   \begin{equation} \label{condition3}
   ( (i-j-1) ! )^{-1}  
   \theta(e_{-1}^{i-j} \circ h) = ( i ! )^{-1}  
   \theta(e_{-1}^i e_j \circ h) \hbox{ for } j+1 \leq i \leq p-1, 0 \leq j \leq j_0.
   \end{equation} 
   Note that as a corollary of (\ref{condition2}) we get
   \begin{equation} \label{condition4}
   \theta(e_j \circ h) = 0 \hbox{ for } 0 \leq j \leq j_0.
   \end{equation}
     \end{proof}

     \section{Constructing self-similarity from virtual endomorphism : the  case when the derivation part is  the Frank algebra or $sl_{n+1}(k)$.}  \label{Frank-section}
   
Let $k$ be a field of characteristic $p > 0 $. Set $X = k[x_1, \ldots, x_n]/ I$, where $I$ is the ideal $(x_1^p, \ldots, x_n^p)$. ${\mathcal Der} X$ is  called the Jacobson-Witt algebra, it is the Lie algebra $W(n, \underline{1})$ from \cite[Ch.~4]{Stradebook}. The Jacobson-Witt algebra is simple unless $n = 1, p = 2$.

     There is a Lie algebra ${\mathcal L}$ defined by Frank ( see \cite{Frank}, \cite{Wilson}) as the vector subspace of the vector space with a basis  $$\{ y_i = \partial / \partial x_i,   y_{i,j} = x_i \partial / \partial x_j , 1 \leq i,j \leq n\}$$ of the elements whose sum of  coordinates of $x_i \partial / \partial x_i$ for $i \leq n$ is 0.   The derived subalgebra $[{\mathcal L},{\mathcal L}]$ is a simple Lie algebra (of non-classical type).

     \begin{theorem} \label{Frank0}
     Let $X = k[x_1, \ldots, x_n]/ (x_1^p, \ldots, x_n^p)$, where $k$ is a field of positive characteristic $p$.
     Let $L$ be a Lie algebra over $k$, $H$ an ideal of $L$ such that : 
     
      a)  the short exact sequence of Lie algebras $H \to L \to L/H$ splits, i.e. $$L = H  \leftthreetimes L_0, \hbox{ where }L_0 \simeq L/H;$$
      
         b) $L_0 $ is the Frank Lie algebra  ${\mathcal L}$ of the Jacobson-Witt algebra ${\mathcal Der}X$. 
         
         Let $$\theta : H \to L$$ be a Lie algebra homomorphism such that for every $h \in H$ :
         
         1.  $\theta(y_i^p \circ h) = 0$ for  $ 1 \leq i \leq n$;
         
         2.    $
   \theta(y_1^{i_1} \ldots y_n^{i_n} y_{k,j} \circ \ h) =  i_k 
   \theta(y_1^{i_1} \ldots y_{k}^{i_k-1} \ldots y_j^{i_j+1} \ldots  y_n^{i_n}  \circ \ h) 
   $ for
   $0 \leq i_1, i_{k-1}, i_{k+1},$ $  \ldots, i_{j-1}, i_{j+1},  \ldots, i_n \leq p-1; 1 \leq i_k \leq p-1, 0 \leq i_j \leq p-2$ and $k \not= j$;

   3.  $
   \theta(y_1^{i_1} \ldots y_n^{i_n} y_{k,j} \circ h) = 0
    $ for  $i_k =0$ or $i_j = p-1$  for $k \not= j$;
      
     4. $
    \theta(y_1^{i_1} \ldots y_n^{i_n} (y_{j,j} - y_{1,1}) \circ h)  = 
   (i_j - i_1) 
   \theta(y_1^{i_1} \ldots y_n^{i_n} \circ h)$ for $ 2 \leq j \leq n-1$.
     
     Then
      there is a homomorphism of Lie algebras
      $$\psi :  L \to L \wr {\mathcal Der } X = (X \otimes L) \leftthreetimes {\mathcal Der}X
     $$
     such that 
     $\theta$ is the virtual endomorphism associated to $\psi$ with respect to the augmentation map $\epsilon : X \to k$ (i.e. $\epsilon(x_i) = 0$ for $ 1 \leq i \leq n$) and 
    the restriction of $\psi$ on $L_0$ is the inclusion of $L_0$ in ${\mathcal Der}X$.
    
     \end{theorem}

     \begin{proof} Write $y_i$ for the element $\partial / \partial x_i \in L_0$ and $y_{i,j}$ for $x_i \partial / \partial x_j \in {\mathcal Der} X$. By Theorem B applied for $H \leftthreetimes (k y_1 \oplus \ldots \oplus k y_n)$    there is  a Lie algebra homomorphism
     $$\psi : H \to X \otimes L$$
     given by 
     $$
\psi(h) =
 \sum_{0 \leq i_1, \ldots, i_n \leq p-1} ( i_1 ! \ldots i_n !)^{-1} x_1^{i_1} \ldots x_n^{i_n}  \otimes 
   \theta(y_1^{i_1} \ldots y_n^{i_n} \circ h)
$$
and
$$
\psi([y_i, h]) = [y_i, \psi(h)] \hbox{ for every } 1 \leq i \leq n.
$$
Observe that we can apply Theorem B because of condition 1 from the statement of Theorem \ref{Frank0}.
It remains to check that
$$
\psi([y_{i,j}, h]) = [y_{i,j}, \psi(h)] \hbox{ for every } 1 \leq j \not= i \leq n
$$
and
$$
\psi([y_{1,1} - y_{i,i}, h]) = [ y_{1,1} - y_{i,i}, \psi(h)] \hbox{ for every } 2 \leq  i \leq n.
$$
By the definition of $\psi$ we have
$$
\psi([y_{k,j}, h]) = \psi(y_{k,j} \circ h) =
 \sum_{0 \leq i_1, \ldots, i_n \leq p-1} ( i_1 ! \ldots i_n !)^{-1} x_1^{i_1} \ldots x_n^{i_n}  \otimes 
   \theta(y_1^{i_1} \ldots y_n^{i_n} y_{k,j} \circ h)
$$
and for $k \not= j$
$$
 [y_{k,j}, \psi(h)]  = 
 \sum_{0 \leq i_1, \ldots, i_n \leq p-1} ( i_1 ! \ldots i_n !)^{-1} y_{k,j}(x_1^{i_1} \ldots x_n^{i_n})  \otimes 
   \theta(y_1^{i_1} \ldots y_n^{i_n} \circ h) =
$$
$$
 \sum_{1 \leq i_j \leq p-1, 0 \leq i_1, \ldots, i_{j-1}, i_{j+1}, \ldots,  i_n \leq p-1} ( i_1 ! \ldots i_k! \ldots (i_j-1)! \ldots i_n !)^{-1} x_1^{i_1} \ldots x_k^{i_k+1} \ldots x_j^{i_j - 1} \ldots x_n^{i_n}  \otimes 
  $$ $$ \theta(y_1^{i_1} \ldots y_n^{i_n} \circ h) =
   $$
   $$
    \sum_{0 \leq i_1, \ldots ,{i_{k-1}}, i_{k+1},  \ldots, {i_{j-1}}, i_{j+1},  \ldots, i_n \leq p-1; 1 \leq i_k \leq p-1, 0 \leq i_j \leq p-2} ( i_1 ! \ldots (i_k-1)! \ldots (i_j)! \ldots i_n !)^{-1}$$ $$ x_1^{i_1} \ldots x_k^{i_k} \ldots x_j^{i_j} \ldots x_n^{i_n}  \otimes 
   \theta(y_1^{i_1} \ldots y_{k}^{i_k-1} \ldots y_j^{i_j+1} \ldots  y_n^{i_n} \circ h). 
   $$
Then conditions 2 and 3 from the statement of the theorem imply that
   $$\psi([y_{k,j}, h]) = [\psi(y_{k,j}), \psi( h)] \hbox{ for } k \not= j.
   $$
   
   Finally consider
   $$
  [y_{j,j}, \psi(h)]  = 
 \sum_{0 \leq i_1, \ldots, i_n \leq p-1} ( i_1 ! \ldots i_n !)^{-1} y_{j,j}(x_1^{i_1} \ldots x_n^{i_n})  \otimes 
   \theta(y_1^{i_1} \ldots y_n^{i_n} \circ h) =
$$ 
$$
 \sum_{0 \leq i_1, \ldots, i_n \leq p-1} ( i_1 ! \ldots i_k! \ldots i_j! \ldots i_n !)^{-1} i_j x_1^{i_1} \ldots  x_j^{i_j} \ldots x_n^{i_n}  \otimes 
   \theta(y_1^{i_1} \ldots y_n^{i_n} \circ h) 
   $$
   Hence for $j \not= 1$ we have
   $$[\psi(y_{j,j} - y_{1,1}), \psi(h)]  = [y_{j,j} - y_{1,1}, \psi(h)]  = $$ $$
    \sum_{0 \leq i_1, \ldots, i_n \leq p-1} ( i_1 ! \ldots i_j! \ldots i_n !)^{-1}(i_j - i_1) x_1^{i_1} \ldots  x_j^{i_j} \ldots x_n^{i_n}  \otimes 
   \theta(y_1^{i_1} \ldots y_n^{i_n} \circ h) $$
Then since
$$
\psi([y_{j,j} - y_{1,1},h]) =  \sum_{0 \leq i_1, \ldots, i_n \leq p-1} ( i_1 ! \ldots i_n !)^{-1} x_1^{i_1} \ldots x_n^{i_n}  \otimes 
   \theta(y_1^{i_1} \ldots y_n^{i_n} (y_{j,j} - y_{1,1}) \circ h), 
    $$
   we deduce that $\psi([y_{j,j} - y_{1,1},h]) = [\psi(y_{j,j} - y_{1,1}),\psi(h)] $ is equivalent to
   $$
    \theta(y_1^{i_1} \ldots y_n^{i_n} (y_{j,j} - y_{1,1}) \circ h)  = 
   (i_j - i_1) 
   \theta(y_1^{i_1} \ldots y_n^{i_n} \circ h).$$
   This is condition 4 from the statement of the theorem.
\end{proof}
     
		 There is an embedding of $sl_{n+1}( \mathbb{C})$ in ${\mathcal Der} ( {\mathbb C}
		[t_1, t_1^{-1}, \ldots, t_n, t^{-1}_n])$ by \cite{M}.
		
		 There is a more general version of embedding  of $sl_{n+1}(k)$ in a Lie algebra of derivations, 
when $char(k)$ does not divide $n+1$ i.e  $sl_{n+1}( k)$ embeds in ${\mathcal Der} X$, where $X = k[x_1, \ldots, x_n] / (x_1^p, \ldots, x_n^p)$ if $p > 0$ and $X =  k[x_1, \ldots, x_n] $ if $p = 0$.
 Let $\widetilde{L}$ be the $k$-span of 
\begin{equation} \label{basis} \{ y_{i,j} \ |  \ 1 \leq i \not= j \leq n+1 \} \cup \{ y_{i,i} - y_{i+1, i+1} \ | \ 1 \leq i \leq n \}, \end{equation}
where
$$
y_{i,j} = x_i \partial / \partial x_j \hbox{ for } 1 \leq i,j \leq n;$$ $$
y_{i, n+1} = x_i \sum_{1 \leq j \leq n} x_j \partial / \partial x_j \hbox{ for } 1 \leq i \leq n;$$ $$
y_{n+1, i} = \partial / \partial x_i \hbox{ for } 1 \leq i \leq n;$$ $$
y_{n+1, n+1} = - \sum_{1 \leq i \leq n} x_i \partial / \partial x_i$$ 
To keep in line with a previous notation we write $y_i$ for 	$y_{n+1, i} = \partial / \partial x_i $.	It is easy to check that   $\widetilde{L} \simeq sl_{n+1}( k)$ provided $char(k)$ does not divide $n+1$
.  In this case $\widetilde{L}$ is the $k$-span of 
$ \{ y_{i,j} \ |  \ 1 \leq i \not= j \leq n+1 \} \cup \{ y_{i,i}  \ | \ 1 \leq i \leq n \}$.

		  In the next result we consider a criterion for self-similar Lie algebra structure, where the derivation part is $\widetilde{L}$.
		  
		 \begin{theorem} \label{sl}
     Let $X = k[x_1, \ldots, x_n]/ (x_1^p, \ldots, x_n^p)$, where $k$ is a field of positive characteristic $p$ such that $p$ does not divide $n+1$.
     Let $L$ be a Lie algebra over $k$, $H$ an ideal of $L$ such that : 
     
      a)  the short exact sequence of Lie algebras $H \to L \to L/H$ splits, i.e. $$L = H  \leftthreetimes L_0, \hbox{ where }L_0 \simeq L/H;$$
      
         b) $L_0= \widetilde{L}$ be  as above.
         
         Let $$\theta : H \to L$$ be a Lie algebra homomorphism such that for every $h \in H$ :
         
         1.  $\theta(y_i^p \circ h) = 0$ for  $ 1 \leq i \leq n$;
         
         2.    $
   \theta(y_1^{i_1} \ldots y_n^{i_n} y_{k,j} \circ \ h) =  i_k 
   \theta(y_1^{i_1} \ldots y_{k}^{i_k-1} \ldots y_j^{i_j+1} \ldots  y_n^{i_n}  \circ \ h) 
   $ for
   $0 \leq i_1, i_{k-1}, i_{k+1},$ $  \ldots, i_{j-1}, i_{j+1},  \ldots, i_n \leq p-1; 1 \leq i_k \leq p-1, 0 \leq i_j \leq p-2$ and $k \not= j$;

   3.  $
   \theta(y_1^{i_1} \ldots y_n^{i_n} y_{k,j} \circ h) = 0
    $ for  $i_k =0$ or $i_j = p-1$  for $k \not= j$;
      
     4. $
    \theta(y_1^{i_1} \ldots y_n^{i_n} y_{j,j} \circ h)  = 
   i_j 
   \theta(y_1^{i_1} \ldots y_n^{i_n} \circ h)$ for $ 1 \leq j \leq n$;

     5. 
   $\theta(y_1^{i_1} \ldots y_n^{i_n} y_{j,n+1} \circ h) = \sum_{1 \leq k \leq n} i_k i_j \theta(y_1^{i_1} \ldots y_j^{i_j - 1} \ldots y_n^{i_n} \circ h)$ for $ i_j \geq 1$ and $\theta(y_1^{i_1} \ldots y_j^0 \ldots  y_n^{i_n} y_{j,n+1} \circ h) = 0$ for $ 1 \leq j \leq n$.

     Then
      there is a homomorphism of Lie algebras
      $$\psi :  L \to L \wr {\mathcal Der } X = (X \otimes L) \leftthreetimes {\mathcal Der}X
     $$
     such that 
     $\theta$ is the virtual endomorphism associated to $\psi$ with respect to the augmentation map $\epsilon : X \to k$ (i.e. $\epsilon(x_i) = 0$ for $ 1 \leq i \leq n$) and 
    the restriction of $\psi$ on $L_0$ is the inclusion of $L_0$ in ${\mathcal Der}X$.

     \end{theorem}

     \begin{proof}  
     Let $\widehat{L}$ be the Frank algebra i.e. $\widehat{L}$ is spanned as $k$-vector space by $\{ y_{i,j} \}_{1 \leq i \not= j \leq n} \cup \{ y_{i,i} - y_{1,1} \}_{2 \leq i \leq n} \cup \{ y_i \}_{1 \leq i \leq n}$. Then by Theorem \ref{Frank0} applied for the Lie algebra $H  \leftthreetimes  \widehat{L}$  there is  a Lie algebra homomorphism
     $$\psi : H \to X \otimes L$$
     defined by 
     $$
\psi(h) =
 \sum_{0 \leq i_1, \ldots, i_n \leq p-1} ( i_1 ! \ldots i_n !)^{-1} x_1^{i_1} \ldots x_n^{i_n}  \otimes 
   \theta(y_1^{i_1} \ldots y_n^{i_n} \circ h),
$$
$$
\psi([y_i, h]) = [y_i, \psi(h)] \hbox{ for every } 1 \leq i \leq n,
$$
$$
\psi([y_{i,j}, h]) = [y_{i,j}, \psi(h)] \hbox{ for } 1 \leq i \not= j \leq n
$$
and
\begin{equation} \label{frank-cond} 
 \psi([y_{j,j} - y_{1,1},h]) = [\psi(y_{j,j} - y_{1,1}),\psi(h)] 
\hbox{ for } 2 \leq j \leq n.
\end{equation}
Observe that we can apply Theorem \ref{Frank0} because conditions 1,2,3 and 4 from the statement of the theorem hold.
Note that for $ 1 \leq j \leq n$ by condition 4 from the statement of Theorem \ref{sl}
   $$ [\psi(y_{j,j}), \psi(h)] =
  [y_{j,j}, \psi(h)] =$$ $$
 \sum_{0 \leq i_1, \ldots, i_n \leq p-1} i_j ( i_1 !  \ldots i_n !)^{-1}  x_1^{i_1} \ldots  x_n^{i_n}  \otimes 
   \theta(y_1^{i_1} \ldots y_n^{i_n} \circ h) =
   $$
   $$
     \sum_{0 \leq i_1, \ldots, i_n \leq p-1} ( i_1 !  \ldots i_n !)^{-1}  x_1^{i_1} \ldots  x_n^{i_n}  \otimes 
   \theta(y_1^{i_1} \ldots y_n^{i_n} y_{j,j} \circ h) =  \psi([y_{j,j}, h]).$$
Finally it remains to check that $ 1 \leq j \leq n$
$$
[y_{j,n+1}, \psi(h)] = \psi([y_{j,n+1},h])$$
We calculate
$$
[y_{j,n+1}, \psi(h)] =
 \sum_{0 \leq i_1, \ldots, i_n \leq p-1} ( i_1 ! \ldots i_n !)^{-1} y_{j,n+1} ( x_1^{i_1} \ldots x_n^{i_n} ) \otimes 
   \theta(y_1^{i_1} \ldots y_n^{i_n} \circ h) = 
$$
$$
\sum_{1 \leq k \leq n}  \sum_{0 \leq i_1, \ldots, i_n \leq p-1} ( i_1 ! \ldots i_n !)^{-1} (x_j x_k \partial / \partial_k)( x_1^{i_1} \ldots x_n^{i_n}  )\otimes 
   \theta(y_1^{i_1} \ldots y_n^{i_n} \circ h) = $$
   $$
   \sum_{1 \leq k \leq n}  \sum_{0 \leq i_1, \ldots, i_{k}, \ldots,  i_n \leq p-1} i_k( i_1 ! \ldots i_n !)^{-1} ( x_1^{i_1}  \ldots x_j^{i_j + 1} \ldots  x_n^{i_n}  )\otimes 
   \theta(y_1^{i_1} \ldots y_n^{i_n} \circ h)  = $$
   and
   $$
   \psi([y_{j, n+1},h]) = 
 \sum_{0 \leq i_1, \ldots, i_n \leq p-1} ( i_1 ! \ldots i_n !)^{-1} ( x_1^{i_1} \ldots x_n^{i_n} ) \otimes 
   \theta(y_1^{i_1} \ldots y_n^{i_n} y_{j,n+1} \circ h).$$
Hence
$$[y_{j,n+1}, \psi(h)] = \psi([y_{j, n+1},h])$$
is equivalent to  condition 5 from the statement of the theorem.
\end{proof}

 \section{Constructing self-similarity from virtual endomorphism : the  case when the derivation part is Heisenberg Lie algebra ${\mathcal H}$ of dimension 3} \label{section-heisenberg}

The Heisenberg Lie algebra ${\mathcal H}$ of dimension 3 has a basis $y_1, y_2, y_3$ and satisfies the relation $[y_1, y_3] = y_2$.

\begin{theorem} \label{heisenberg} 

     Let $X = k[x_1, x_2]/I$, where $I = (x_1^p, x_2^p)$ if $char(k) = p >0$ and $I = 0$ if $char(k) = 0$.
   Let $L$ be a Lie algebra over a field $k$, $H$ an ideal of $L$ such that   $$L = H  \leftthreetimes L_0, \hbox{ where }L_0 \simeq {\mathcal H}.$$  
Let $$\theta : H \to L$$ be a Lie algebra homomorphism such that :

1.  there is $m  \in {\mathbb N}$ such that for every $b \in \{ y_1, y_2 \}$ and $h \in H$ we have $\theta(b^m \circ h) = 0$. If $char(k) = p >0$ we furthermore have that $m = p$. 
 
2.  $\theta(y_1 y_3 \circ h) = \theta(y_2 \circ h)$ and $\theta(y_3 \circ h) = 0$  for every $h \in H$.
    
     Then
      there is a homomorphism of Lie algebras
      $$\psi :  L \to L \wr {\mathcal Der } X = (X \otimes L) \leftthreetimes {\mathcal Der}X
     $$
     given by 
 \begin{equation} \label{nilp}
\psi(h) =
 \sum_{0 \leq i_1, i_2 \leq m_0 = m-1} ( i_1 ! i_2 !)^{-1} x_1^{i_1} x_2^{i_2}  \otimes 
   \theta(y_1^{i_1} y_2^{i_2} \circ h) \hbox{ for } h \in H,
\end{equation}    
 where     $\theta$ is the virtual endomorphism associated to $\psi$ with respect to the augmentation map $\epsilon : X \to k$ i.e. $\epsilon$ is the $k$-algebra homomorphism  defined by $\epsilon(x_i) = 0$ for  $1 \leq i \leq 2$ and $\psi(y_i) = \partial / \partial {x_i}$ for  $ 1 \leq i \leq 2$ and
$\psi(y_3) = x_1 \partial / \partial x_2$;  
 \end{theorem}

\begin{proof} Since the subalgebra $L_1$ of $L_0$ generated by $y_1, y_2$ is abelian we can apply Theorem B for $L_2 = H  \leftthreetimes L_1$ and deduce that there is homomorphism of Lie algebras
$$
\psi : L_2 \to L_2 \wr {\mathcal Der} X
$$
given by (\ref{nilp}). Note that we can apply Theorem B since condition 1 from the statement of the theorem  holds.

Since $L = L_2 \leftthreetimes k y_3$ and $\psi(y_3) = x_1 \partial / \partial x_2$  it remains to show that
$[\psi(y_3), \psi(h)] = \psi([y_3, h])$. As before we denote by $a \circ b$ the commutator $[a,b]$. Note that by (\ref{nilp})
$$
 \psi([y_3, h]) = \psi(y_3 \circ h) =
 \sum_{0 \leq i_1, i_2 \leq m_0 = m-1} ( i_1 ! i_2 !)^{-1} x_1^{i_1} x_2^{i_2}  \otimes 
   \theta(y_1^{i_1} y_2^{i_2} y_3 \circ h).
$$
On other hand for
$$
[\psi(y_3), \psi(h)] = [x_1 \partial / \partial x_2, 
 \sum_{0 \leq i_1, i_2 \leq m_0} ( i_1 ! i_2 !)^{-1} x_1^{i_1} x_2^{i_2}  \otimes 
   \theta(y_1^{i_1} y_2^{i_2} \circ h) ] =$$ $$
 \sum_{0 \leq i_1, i_2 \leq m_0} ( i_1 ! i_2 !)^{-1}  (x_1 \partial / \partial x_2) (x_1^{i_1} x_2^{i_2})  \otimes 
   \theta(y_1^{i_1} y_2^{i_2} \circ h) =$$ $$ \sum_{0 \leq i_1, i_2 \leq m_0} ( i_1 ! i_2 !)^{-1} i_2 x_1^{i_1+1} x_2^{i_2-1}  \otimes 
   \theta(y_1^{i_1} y_2^{i_2} \circ h)  = 
$$
$$
\sum_{1 \leq j_1 \leq m_0,0 \leq j_2 \leq m_0-1} ( (j_1-1) ! j_2 !)^{-1}  x_1^{j_1} x_2^{j_2}  \otimes 
   \theta(y_1^{j_1-1} y_2^{j_2+1} \circ h). $$
Thus $[\psi(y_3), \psi(h)] = \psi([y_3, h]) $   if and only if
\begin{equation} \label{nilpo2}
  \theta(y_1^{i_1} y_2^{i_2} y_3 \circ h) = i_1
   \theta(y_1^{i_1-1} y_2^{i_2+1} \circ h) \hbox{ for } i_1 \geq 1
\end{equation} 
\begin{equation} \label{nilpo3}
\theta(y_2^{i_2} y_3 \circ h) = 0 \hbox{ for } i_2 \geq 0.
\end{equation}
and
\begin{equation} \label{nilpo4}
\theta( y_1^{i_1}  y_2^{m-1} y_3 \circ h) = 0 \hbox{ for } i_1 \geq 0.
\end{equation}
Observe that since $y_2$ is a central element in ${\mathcal H}$ (\ref{nilpo2}) is equivalent to (\ref{nilpo2}) for $i_2 = 0$ and (\ref{nilpo3}) is equivalent to (\ref{nilpo3}) for $i_2 = 0$. We will show that for  (\ref{nilpo2}) to hold for any $i_1 \geq 1, i_2 = 0$ it suffices it holds for $i_1 = 1, i_2 = 0$ and (\ref{nilpo3}) holds for $i_2 = 0$.
To show this suppose that (\ref{nilpo2}) holds for $i_1 - 1, i_2 = 0$ and we will show that the same holds for $i_1$. Then for $i_1 \geq 2$ we have
$$
 \theta(y_1^{i_1} y_3 \circ h) =  \theta(y_1^{i_1-1} (y_1 y_3 - y_3 y_1) \circ h) + \theta(y_1^{i_1-1}y_3 \circ (y_1 \circ h))  = 
$$
$$ \theta(y_1^{i_1-1} y_2 \circ h) + (i_1 - 1)  \theta(y_1^{i_1-2}y_2 \circ (y_1 \circ h)) =
$$
$$
\theta(y_1^{i_1-1} y_2 \circ h) + (i_1 - 1)  \theta(y_1^{i_1-1}y_2 \circ h) =
i_1 
\theta(y_1^{i_1-1} y_2 \circ h),
$$
as required.

Observe that for  $ i_1 \geq 1$ (\ref{nilpo4}) follows from (\ref{nilpo2}) and condition 1 from the statement of the theorem. Indeed
$$
\theta( y_1^{i_1}  y_2^{m-1} y_3 \circ h) = \theta( y_1^{i_1}   y_3 \circ ( y_2^{m-1} \circ h)) = i_1 \theta( y_1^{i_1-1}   y_2 \circ ( y_2^{m-1} \circ h)) = $$ $$i_1 \theta( y_1^{i_1-1}  y_2^{m} \circ h) = i_1  \theta(   y_2^{m} \circ ( y_1^{i_1-1} \circ h)) = 0.
$$
Finally  since $\theta (y_3 \circ H) = 0$ we get
$$
\theta( y_2^{m-1} y_3 \circ h) = \theta( y_3 \circ (y_2^{m-1} \circ h)) = 0.
$$ 
Thus (\ref{nilpo4}) is  a redundant relation i.e. follows from (\ref{nilpo2}), (\ref{nilpo3}) and condition 1.
\end{proof}

\begin{theorem} \label{heisenberg2}
Let $X = k[x_1, x_2, x_4]/I$,  where $I = (x_1^p, x_2^p, x_4^p)$ if $char(k) = p >0$ and $I = 0$ if $char(k) = 0$.
    Let $L$ be a Lie algebra over $k$, $H$ an ideal of $L$ such that   $$L = H  \leftthreetimes L_0, \hbox{ where }L_0 \simeq {\mathcal H} \oplus k y_4.$$  
Let $$\theta : H \to L$$ be a Lie algebra homomorphism such that :

1.   there is $m  \in {\mathbb N}$ such that for every $b \in \{ y_1, y_2, y_4 \}$ and $h \in H$ we have $\theta(b^m \circ h) = 0$. If $char(k) = p >0$ we furthermore have that $m = p$. 
 
2.  $\theta(y_1 y_3 \circ h) = \theta(y_2 \circ h)$ and $\theta(y_3 \circ h) = 0$ for every $h \in H$.
    
     Then
      there is a homomorphism of Lie algebras
      $$\psi :  L \to L \wr {\mathcal Der } X = (X \otimes L) \leftthreetimes {\mathcal Der}X
     $$
     given by 
 \begin{equation} \label{nilp12}
\psi(h) =
 \sum_{0 \leq i_1, i_2, i_4 \leq m_0 = m-1} ( i_1 ! i_2 ! i_4!)^{-1} x_1^{i_1} x_2^{i_2} x_4^{i_4} \otimes 
   \theta(y_1^{i_1} y_2^{i_2} y_4^{i_4} \circ h),
\end{equation}    
 where $\theta$ is the virtual endomorphism associated to $\psi$ with respect to the augmentation map $\epsilon : X \to k$ i.e.  $\epsilon(x_i) = 0$ for  $i \in \{ 1,2, 4 \}$ and
$\psi(y_i) = \partial / \partial {x_i}$ for $i \in \{ 1,2, 4 \}$ and
$\psi(y_3) = x_1 \partial / \partial x_2$;  
 \end{theorem}

\begin{proof} Since the subalgebra $L_1$ of $L_0$ generated by $y_1, y_2, y_4$ is abelian we can apply Theorem B for $L_2 = H  \leftthreetimes L_1$ and deduce that there is homomorphism of Lie algebras
$$
\psi : L_2 \to L_2 \wr {\mathcal Der} X
$$
given by (\ref{nilp12}).
Since $L = L_2 \leftthreetimes k y_3$ and $\psi(y_3) = x_1 \partial / \partial x_2$ is defined it remain to show that
$[\psi(y_3), \psi(h)] = \psi([y_3, h])$. As in the proof of Theorem \ref{heisenberg} just adding $x_4$ and $y_4$ in the calculations we get that
 $\psi : L \to L \wr {\mathcal H} \subset L \wr {\mathcal Der}X$ is a homomorphism if and only if for $h \in H$ we have 
\begin{equation} \label{nilpo21}
  \theta(y_1^{i_1} y_2^{i_2} y_4^{i_4} y_3  \circ h) = i_1
   \theta(y_1^{i_1-1} y_2^{i_2+1} y_4^{i_4} \circ h) \hbox{ for } i_1 \geq 1, i_2, i_4 \geq 0 
\end{equation} 
\begin{equation} \label{nilpo31}
\theta(y_2^{i_2} y_4^{i_4} y_3 \circ h) = 0 \hbox{ for } i_2, i_4 \geq 0.
\end{equation}
and
\begin{equation} \label{nilpo41}
\theta( y_1^{i_1}  y_2^{m-1}  y_4^{i_4} y_3 \circ h) = 0 \hbox{ for } i_1, i_4 \geq 0.
\end{equation}
Since $y_4$ is a central element in $L_0$ the above equations are equivalent to
\begin{equation} 
  \theta(y_1^{i_1} y_2^{i_2} y_3  \circ h_0) = i_1
   \theta(y_1^{i_1-1} y_2^{i_2+1} \circ h_0) \hbox{ for } i_1 \geq 1, i_2\geq 0 
\end{equation} and
\begin{equation} 
\theta(y_2^{i_2} y_3 \circ h_0) = 0 \hbox{ for } i_2\geq 0.
\end{equation}
for $h_0 = y_4^{i_4} \circ h \in H$, which are precisely the conditions (\ref{nilpo2}) and (\ref{nilpo3}) considered before.
\end{proof}

\section{Self-similar structure of $sl_{n+1}(k)$} \label{embed-new}

		 Recall that  if $char(k)$ does not divide $n+1$,  $sl_{n+1}( k)$ embeds in ${\mathcal Der} X$, where $X = k[x_1, \ldots, x_n] / (x_1^p, \ldots, x_n^p)$ if $p > 0$  and $X = k[x_1, \ldots, x_n]$ if $p = 0$. 
  The $k$-span of 
$\{ y_{i,j} \ |  \ 1 \leq i \not= j \leq n+1 \} \cup \{ y_{i,i} - y_{i+1, i+1} \ | \ 1 \leq i \leq n \}$ is a copy of $sl_{n+1}(k)$ and coincides with the $k$-span of 
$\{ y_{i,j} \ |  \ 1 \leq i \not= j \leq n+1 \} \cup \{ y_{i,i} \ | \ 1 \leq i \leq n \}$ . 
In ``matrix'' notation  i.e. $E_{a,b} E_{c,d} = \delta_{b,c} E_{a,d}$ we have :  
$$
E_{i,j} = - x_j \partial / \partial x_i \hbox{ for } 1 \leq i,j \leq n;$$ $$
E_{n+1,i} = - x_i \sum_{1 \leq j \leq n} x_j \partial / \partial x_j \hbox{ for } 1 \leq i \leq n;$$ $$
E_{i,n+1} = \partial / \partial x_i \hbox{ for } 1 \leq i \leq n \hbox{ and }
E_{n+1, n+1} =  \sum_{1 \leq i \leq n} x_i \partial / \partial x_i.$$

We observe in the following lemma  that there is an obvious  self-similar structure on $s(n+1,k)$, which is a diagonal embedding.

\begin{lemma}Let $k$ be a field such that $char(k)$ does not divide $n+1$. Then 
there is a Lie algebra monomorphism
$$\psi : sl_{n+1}(k) \to (X \otimes sl_{n+1}( k)) \leftthreetimes {\mathcal Der}(X),$$
where $X = k[x_1, \ldots, x_n]$ if $char(k) = 0$ and $X = k[x_1, \ldots, x_n]/ (x_1^p, \ldots, x_n^p)$ if $char(k) = p > 0$. Furthermore $sl_{n+1}(k)$ is a faithful $\psi$-self-similar Lie algebra. 
\end{lemma}

\begin{proof}
We identify $sl_{n+1}(k)$ with the subalgebra $\widetilde{L}$ of ${\mathcal Der}(X)$. Define $$ \psi(a) = 1 \otimes a + a \in (X \otimes \widetilde{L}) \leftthreetimes {\mathcal  Der}(X)$$
Note that this is a diagonal embedding, since $(1 \otimes \widetilde{L}) \leftthreetimes \widetilde{L} \simeq \widetilde{L} \oplus \widetilde{L}$.

Since $ \pi \psi : sl_{n+1}(k) \to {\mathcal Der}(X)$ is a monomorphism, $\psi$ is a monomorphism and $sl_{n+1}(k)$ acts faithfully on $X$, hence acts faithfully on $\cup_{m \geq 1} X^{\otimes m}$.
\end{proof}

Theorem C gives a more interesting  self-similar structure of $sl_{n+1}(k)$ which is not a diagonal embedding and is not the obvious embedding inside ${\mathcal Der}(X)$ or $1 \otimes sl_{n+1}(k)$. It is inspired  by the examples of \cite{Bartholdi}, where always one of the generators goes to $\partial / \partial x$.

\medskip		  
		  {\bf Proof of Theorem C}

We have to check that $\psi$ is a homomorphism of Lie algebras.  Since $ \pi \psi : sl_{n+1}(k) \to {\mathcal Der}(X)$ is a monomorphism, $\psi$ is a monomorphism.  And again since $\pi \psi$ is monomorphism $sl_{n+1}(k)$ acts faithfully on $X$, hence acts faithfully on $\cup_{m \geq 1} X^{\otimes m}$.	  
		  
1. 		  Note that the $k$-span $L_1$ of $\{ E_{i,j} | 1 \leq i,j \leq n \}$ is a Lie subalgebra of ${\mathcal Der} X$ and $\psi$ is defined on it as the ``diagonal'' map $L_1 \to (1 \otimes L_1 ) \oplus L_1$, hence the restriction of $\psi$ on $L_1$ is a homomorphism.

	2. Note that $[E_{i,n+1}, E_{j,n+1}] = 0$ for $ 1 \leq i,j \leq n$ and
	$$
	[\psi(E_{i,n+1}), \psi(E_{j, n+1})] = [\partial / \partial x_i, \partial / \partial x_j] = 0 = \psi(	  [E_{i,n+1}, E_{j,n+1}]).$$
	
	3. For $1 \leq a,b \leq n$ and $ 1 \leq i \leq n$ we have
	$$
	[\psi(E_{a,b}) , \psi(E_{i, n+1})] = [ 1 \otimes E_{a,b} + E_{a,b}, E_{i, n+1}] = - (\partial / \partial x_i)(1) \otimes E_{a,b} + [ E_{a,b}, E_{i, n+1}] = $$ $$ [E_{a,b}, E_{i, n+1}] \in {\mathcal Der}(X).
	$$
	On other hand $ [E_{a,b}, E_{i, n+1}] = \delta_{b,i} E_{a, n+1}
		 $ where $\delta_{b,i}$ is the Kroneker symbol, hence
		 $$\psi([E_{a,b}, E_{i, n+1}]) = \delta_{b,i} \psi(E_{a, n+1})  = \delta_{b,i} E_{a, n+1} = [E_{a,b}, E_{i, n+1}] = [\psi(E_{a,b}) , \psi(E_{i, n+1})] .$$

		  		  4. Note that $[E_{n+1,i}, E_{j,n+1}] = \delta_{i,j} E_{n+1, n+1} - E_{j,i}$ where $ 1 \leq i,j \leq n$. 
		  		  And hence
		  $$\psi([E_{n+1,i}, E_{j,n+1}]) = \psi(\delta_{i,j} E_{n+1, n+1} - E_{j,i}) = $$ $$1 \otimes (\delta_{i,j} E_{n+1, n+1} - E_{j,i}) + (\delta_{i,j} E_{n+1, n+1} - E_{j,i}).$$
		  		  On other hand
		  $$[\psi(E_{n+1,i}), \psi(E_{j,n+1})] =  [\sum_{1 \leq s \leq n} x_s \otimes b_{i,s} + 1 \otimes E_{n+1,i} + E_{n+1,i}, \partial/ \partial x_j] =$$ $$ [\sum_{1 \leq s \leq n} x_s \otimes b_{i,s}, \partial/ \partial x_j] + [1 \otimes E_{n+1,i}, \partial/ \partial x_j] + [E_{n+1,i}, \partial/ \partial x_j] = $$ 
		  $$- \sum_{1 \leq s \leq n} (\partial/ \partial x_j)(x_s) \otimes b_{i,s} + 0 + [E_{n+1,i},E_{j, n+1}] = - 1 \otimes b_{i,j} + \delta_{i,j} E_{n+1, n+1} - E_{j,i}.$$
		  Hence 
		  $ [\psi(E_{n+1,i}), \psi(E_{j,n+1})] = \psi([E_{n+1,i}, E_{j,n+1}]) $  is equivalent to
		   $$b_{i,j} = E_{j,i} - \delta_{i,j} E_{n+1, n+1}.$$
		   
		  5. Note that for $1 \leq i\not= j \leq n$ we have that $[E_{n+1,i}, E_{n+1,j}] = 0$. On other hand
		  \begin{equation} \label{*}
		  [\psi(E_{n+1,i}), \psi(E_{n+1,j})] =\end{equation} $$ [\sum_{1 \leq s \leq n} x_s \otimes b_{i,s} + 1 \otimes E_{n+1, i} + E_{n+1,i},\sum_{1 \leq t \leq n} x_t \otimes b_{j,t} + 1 \otimes E_{n+1, j} + E_{n+1,j}] = 
		  $$
		  $$
		  [\sum_{1 \leq s \leq n} x_s \otimes b_{i,s},\sum_{1 \leq t \leq n} x_t \otimes b_{j,t}] + [ 1 \otimes E_{n+1, i} + E_{n+1,i}, 1 \otimes E_{n+1, j} + E_{n+1,j}]  +$$ $$ [\sum_{1 \leq s \leq n} x_s \otimes b_{i,s}, 1 \otimes E_{n+1, j} + E_{n+1,j}] + [1 \otimes E_{n+1, i} + E_{n+1,i},\sum_{1 \leq t \leq n} x_t \otimes b_{j,t}] =$$
		  $$\sum_{1 \leq s,t \leq n} x_s x_t \otimes [b_{i,s}, b_{j,t}] + 0
		  + \sum_{1 \leq s \leq n} x_s \otimes [b_{i,s}, E_{n+1,j}] - \sum_{1 \leq s \leq n} E_{n+1,j} (x_s) \otimes b_{i,s} +$$ $$ \sum_{1 \leq t \leq n} x_t \otimes [E_{n+1,i}, b_{j,t}] + \sum_{1 \leq t \leq n} E_{n+1,i}(x_t) \otimes b_{j,t}.
		  $$ Note that
		$$
		[b_{i,s}, b_{j,t}] = [E_{s,i} - \delta_{s,i} E_{n+1, n+1}, E_{t,j} - \delta_{t,j} E_{n+1, n+1}] =
		 [E_{s,i}, E_{t,j}] =
		$$
		$$
		\delta_{i,t}  E_{s,j} - \delta_{j,s} E_{t,i}
		  $$
		  and
		  $$[b_{i,s}, E_{n+1,j}] = [E_{s,i} - \delta_{s,i} E_{n+1, n+1}, E_{n+1,j}] = [E_{s,i}, E_{n+1,j}] + [ - \delta_{s,i} E_{n+1, n+1}, E_{n+1,j}]=
		  $$
		  $$-  \delta_{s,j} E_{n+1,i}  - \delta_{s,i} E_{n+1,j}. 
		  $$
		  Similarly
		  $$ [E_{n+1,i}, b_{j,t}]  = - [b_{j,t}, E_{n+1,i}] = \delta_{t,i} E_{n+1,j}  + \delta_{t,j} E_{n+1,i}. 
		  $$
		  Then \begin{equation} \label{chuva1} \sum_{1 \leq s \leq n} x_s \otimes [b_{i,s}, E_{n+1,j}] = \end{equation} $$\sum_{1 \leq s \leq n} x_s \otimes ( - \delta_{s,j} E_{n+1,i} - \delta_{s,i} E_{n+1,j}) =  - x_j \otimes E_{n+1,i} - x_i \otimes E_{n+1,j},$$
		  \begin{equation} \label{chuva2}  \sum_{1 \leq t \leq n} x_t \otimes [E_{n+1,i}, b_{j,t}] = \end{equation} $$ \sum_{1 \leq t \leq n} x_t \otimes ( \delta_{t,i} E_{n+1,j}  + \delta_{t,j} E_{n+1,i}) = x_j \otimes  E_{n+1,i} + x_i \otimes  E_{n+1,j},$$
		  \begin{equation} \label{chuva3} \sum_{1 \leq s,t \leq n} x_s x_t \otimes [b_{i,s}, b_{j,t}]  =
		  \sum_{1 \leq s,t \leq n} x_s x_t \otimes (\delta_{i,t}  E_{s,j} - \delta_{j,s} E_{t,i})=   \end{equation}
		  $$
		  x_j x_i  \otimes ( E_{j,j} - E_{i,i}) +\sum_{1 \leq s \not= j \leq n}  x_s x_i \otimes E_{s,j} + \sum_{1 \leq t \not=i \leq n} x_j x_t \otimes (- E_{t,i}).$$ Since 
		 $E_{n+1,j} (x_s) = - x_j x_s$	  
		  \begin{equation} \label{chuva4} \sum_{1 \leq s \leq n} E_{n+1,j} (x_s) \otimes b_{i,s} = - \sum_{1 \leq s \leq n} ( x_j x_s) \otimes b_{i,s} = \end{equation} $$ - \sum_{1 \leq s \leq n} ( x_j x_s) \otimes (E_{s,i} - \delta_{s,i} E_{n+1, n+1}) = x_j x_i \otimes ( E_{n+1, n+1} - E_{i,i})  - \sum_{1 \leq s \not= i  \leq n} ( x_j x_s) \otimes E_{s,i}
		  $$ and similarly
	  \begin{equation} \label{chuva5}
		  \sum_{1 \leq t \leq n} E_{n+1,i}(x_t) \otimes b_{j,t} = x_i x_j \otimes (E_{n+1, n+1}- E_{j,j}) -  \sum_{1 \leq t \not= j \leq n} ( x_ix_t) \otimes E_{t,j}.
		  \end{equation}
		  Combining (\ref{chuva1}),  (\ref{chuva2}),  (\ref{chuva3}),  (\ref{chuva4}),  (\ref{chuva5}) with  (\ref{*}), we obtain
		  $$
		  [\psi(E_{n+1,i}), \psi(E_{n+1,j})] = \sum_{1 \leq s,t \leq n} x_s x_t \otimes [b_{i,s}, b_{j,t}] 
		  + \sum_{1 \leq s \leq n} x_s \otimes [b_{i,s}, E_{n+1,j}] $$ $$- \sum_{1 \leq s \leq n} E_{n+1,j} (x_s) \otimes b_{i,s} + \sum_{1 \leq t \leq n} x_t \otimes [E_{n+1,i}, b_{j,t}] + \sum_{1 \leq t \leq n} E_{n+1,i}(x_t) \otimes b_{j,t}=
		  $$
		  $$ 
		  x_j x_i  \otimes ( E_{j,j} - E_{i,i}) +\sum_{1 \leq s \not= j \leq n}  x_s x_i \otimes E_{s,j} + \sum_{1 \leq t \not=i \leq n} x_j x_t \otimes (- E_{t,i})$$ $$
		  - x_i \otimes E_{n+1,j}- x_j \otimes E_{n+1,i}$$ $$ - x_j x_i \otimes (E_{n+1, n+1}- E_{i,i})  + \sum_{1 \leq s \not= i  \leq n} ( x_j x_s) \otimes E_{s,i} + x_j \otimes  E_{n+1,i} +   x_i \otimes  E_{n+1,j} $$ $$+ x_i x_j \otimes (E_{n+1, n+1}- E_{j,j}) -  \sum_{1 \leq t \not= j \leq n} ( x_ix_t) \otimes E_{t,j}=$$
		  $$ 
		  \sum_{1 \leq s \not= j \leq n}  x_s x_i \otimes E_{s,j} + \sum_{1 \leq t \not=i \leq n} x_j x_t \otimes (- E_{t,i})  + \sum_{1 \leq s \not= i  \leq n} ( x_j x_s) \otimes E_{s,i}  -  \sum_{1 \leq t \not= j \leq n} ( x_ix_t) \otimes E_{t,j}=$$ $$ 0 = \pi(0) =  \psi([E_{n+1,i}, E_{n+1,j}]).$$

		  6.  For  $ 1 \leq a,b,i \leq n$ we have
		  \begin{equation} \label{**}
		  [\psi(E_{a,b}), \psi(E_{n+1, i})] = \end{equation} $$ [ 1 \otimes E_{a,b} + E_{a,b}, \sum_{1 \leq j \leq n} x_j \otimes b_{i,j} + 1 \otimes E_{n+1, i} + E_{n+1,i}]=
		  $$
		  $$ [ 1 \otimes E_{a,b} , \sum_{1 \leq j \leq n} x_j \otimes b_{i,j}
		  ] + [E_{a,b}, \sum_{1 \leq j \leq n} x_j \otimes b_{i,j}] + [ 1 \otimes E_{a,b}, 1 \otimes E_{n+1, i}] +
		  $$
		  $$[ E_{a,b},  1 \otimes E_{n+1, i}] + [ 1 \otimes E_{a,b}, E_{n+1,i}]  + [  E_{a,b} ,E_{n+1,i}] =
		  $$
		  $$\sum_{1 \leq j \leq n} x_j \otimes [E_{a,b}, b_{i,j}]
		  +  \sum_{1 \leq j \leq n} E_{a,b}(x_j) \otimes b_{i,j} + 1 \otimes [ E_{a,b},  E_{n+1, i}] +
		  [E_{a,b}, E_{n+1,i}] =$$
		  $$\sum_{1 \leq j \leq n} x_j \otimes  [E_{a,b}, E_{j,i} - \delta_{j,i} E_{n+1, n+1}] + \sum_{1 \leq j \leq n} E_{a,b}(x_j) \otimes(E_{j,i} - \delta_{j,i} E_{n+1, n+1}) + $$
		  $$ 1 \otimes (  - \delta_{i,a} E_{n+1,b}) + (  - \delta_{i,a} E_{n+1,b}).$$
		  Note that $E_{a,b}(x_j) = - \delta_{a,j} x_b$, $[E_{a,b}, E_{n+1,i}] = - \delta_{i,a} E_{n+1, b}$ and hence by (\ref{**}) 
		  $$
		  [\psi(E_{a,b}), \psi(E_{n+1, i})] = \sum_{1 \leq j \leq n} x_j \otimes  [E_{a,b}, E_{j,i} - \delta_{j,i} E_{n+1, n+1}]$$ $$ - \sum_{1 \leq j \leq n} \delta_{a,j} x_b \otimes(E_{j,i} - \delta_{j,i} E_{n+1, n+1}) +  1 \otimes ( - \delta_{i,a} E_{n+1,b})  - \delta_{i,a} E_{n+1,b}=$$
		  $$ \sum_{1 \leq j \leq n} x_j \otimes  [E_{a,b}, E_{j,i}] - x_b \otimes(E_{a,i} - \delta_{a,i} E_{n+1, n+1}) +  1 \otimes ( - \delta_{i,a} E_{n+1,b})  - \delta_{i,a} E_{n+1,b}=$$
		  $$ \sum_{1 \leq j \not= b \leq n} x_j \otimes  [E_{a,b}, E_{j,i}] - x_b \otimes( \delta_{a,i} E_{b,b}  - \delta_{a,i} E_{n+1, n+1}) +  1 \otimes ( - \delta_{i,a} E_{n+1,b})  - \delta_{i,a} E_{n+1,b}=$$
		  $$ - \delta_{a,i} (\sum_{1 \leq j \not= b \leq n} x_j \otimes E_{j,b}  + x_b \otimes(E_{b,b}  - E_{n+1, n+1}) +  1 \otimes (  E_{n+1,b})  + E_{n+1,b}).$$
		  On other hand
		  $[E_{a,b}, E_{n+1,i}] = - \delta_{i,a} E_{n+1,b}$ and hence
		  $$ \psi([E_{a,b}, E_{n+1,i}]) = - \delta_{i,a} \psi ( E_{n+1,b}) =  - \delta_{i,a} ( \sum_{1 \leq j \leq n} x_j \otimes b_{b,j} +  1 \otimes E_{n+1,b} + E_{n+1,b})=$$
		  $$- \delta_{i,a} (\sum_{1 \leq j \leq n} x_j \otimes (E_{j,b} - \delta_{j,b} E_{n+1, n+1}) + 1 \otimes E_{n+1,b} + E_{n+1,b})=
		  $$
		  $$- \delta_{i,a} (\sum_{1 \leq j \not= b \leq n} x_j \otimes E_{j,b} + x_b \otimes (E_{b,b} - E_{n+1, n+1}) +  1 \otimes E_{n+1,b} + E_{n+1,b})= [\psi(E_{a,b}), \psi(E_{n+1, i})] .
		  $$

      \section{Some examples of abelian and nilpotent self-similar Lie algebras} \label{abelian}
     
     \begin{lemma} Let $Q$ be an abelian Lie algebra with a basis $\{ a_i \}_{i \in I}$, where $I = \{ 1,2,\ldots, n \}$  or $I = \{ 1,2,\ldots, n, n+1, \ldots \}$ is countable.  Then $L$ is a faithful self-similar Lie algebra.
     \end{lemma}
     
     \begin{proof} 

We can define $\theta : H \to L$, where $H$ is the span $\{a_i \mid i \geq 2 \}$ and $\theta(a_i) = a_{i-1}$. By Theorem B $L$ is a self-similar Lie algebra, furthermore by Theorem A b) $L$ is a faithful self-similar Lie algebra.
     \end{proof}

\begin{lemma} Let $Q$ be the Lie algebra over a field $k$  with basis $\{ a,b, [a,b] \}$ such that $[a,b]$ is a central element.  Let $\theta : H \to Q$ be the homomorphism of Lie algebras given by
$\theta(b) = a$ and $\theta([a,b]) = a + [a,b]$, where $H$ is the linear span of $b$ and $[a,b]$. Then $Q$ is a faithful $\psi$-self-similar Lie algebra with virtual endomorphism $\theta$ and $X = k[x]/ (x^p)$ if $char(k) = p > 0$ and $X = k[x]$ if $char(k) = 0$.
 \end{lemma}
\begin{proof} Note that $Q$ is nilpotent of class 2, so $a^2 \circ Q = [a,[a, Q]] = 0$ and $Q = H  \leftthreetimes k a$, thus by Theorem B $Q$ is a $\psi$-self-similar Lie algebra with virtual endomorphism $\theta$ and $\psi(a) = \partial / \partial x$. By Theorem A b) $Q$ is a faithful $\psi$-self-similar Lie algebra.
\end{proof}

Let $L$ be the Lie subalgebra (over $\mathbb{F}_p$) of $gl_3(\mathbb{F}_p[x])$ that contains only matrices whose entries below and on the main diagonal are 0. Denote by $e_{i,j}$ the matrix whose only non-zero entry is 1 and is at place $(i,j)$.
For $y_1 = e_{1,2}$, $y_3 = e_{2,3}$, $y_2 = e_{1,3}$ and $y_4 = x e_{1,3}$ we have that the span $L_0$ of $y_1, \ldots, y_4$ is a direct sum of the one dimensional Lie algebra spanned by $y_4$ and  the Heisenberg Lie algebra ${\mathcal H}$ spanned by $y_1, y_2, y_3$. Then
$L = H \leftthreetimes L_0$ where $a = (a_{i,j}) \in H$ if $a_{i,j} \in x^{j-i} \mathbb{F}_p[x]$ for $1 \leq i,j \leq 3$. Note that $L$ is an infinite dimensional nilpotent of class 2 Lie algebra.
   
   \begin{prop}
Let $L$ be the Lie subalgebra (over $\mathbb{F}_p$) of $gl_3(\mathbb{F}_p[x])$ that contains only matrices whose entries below and on the main diagonal are 0. Let  $$
\theta : H \to L
$$
be the homomorphism 
 given by $\theta(a) = b$, where $a = (a_{i,j})$ and $b = (b_{i,j})$ with $b_{i,j} = x^{i-j} a_{i,j}$. Then $L$ is a faithful self-similar Lie algebra with virtual endomorphism $\theta$.
\end{prop}

\begin{proof}Since $L = H \leftthreetimes L_0$, where $L_0 \simeq {\mathcal H} \oplus \mathbb{F}_p$ we can apply  Theorem \ref{heisenberg2}. The faithfulness of the constructed self-similarity follows from Theorem A b) since $H$ does not contain $\theta$-invariant ideal of $L$.
\end{proof}
   \section{An example of a self-similar metabelian Lie algebra of homological type $FP_n$} \label{fpn}
   
   \subsection{Preliminaries on finitely presented metabelian Lie algebras}
   \label{sub-sec-BG}
   
   Recall that a Lie algebra $L$ is finitely presented with a finite presentation $\langle X \mid R \rangle$ if $L \simeq F(X) / (R)$ where $F(X)$ is the free Lie algebra with a free basis $X$ and $(R)$ is the ideal of $F(X)$ generated by $R$.
    
    Let $A \to L \to Q$ be a short exact sequence of Lie algebras over a field $k$, where $L$ is finitely generated and both $A$ and $Q$ are abelian. 
      Let $U(Q)$ be the universal enveloping algebra of $Q$. We view $A$ as a left $U(Q)$-module via the adjoint action i.e. 
      $$
      q \circ a = [b, a] \hbox{ where } a\in A, b \in L \hbox{ such that } \pi(b) = q
      $$
      where $\pi : L \to Q$.
      Since $L$ is finitely generated as a Lie algebra, $A$ is a finitely generated $U(Q)$-module. Consider the $k$-algebra
      $$
      R = U(Q)/ ann_{U(Q)} A.
      $$
      Let $k[[t]]$ be the ring of power series  and  let $K = \cup_{j \geq 1} t^{-j}k[[t]]$ be its field of fractions. Then the Bryant-Groves invariant $\Delta(A,Q)$ is defined by 
      $$
      \Delta(A,Q) = \{ [\varphi] \mid \varphi \hbox{ can be extended to a }k\hbox{-algebra homomorphism }\widetilde{\varphi} : R \to K  \}
      $$
      where $\varphi : Q \to K$ is a $k$-linear map
      and $[\varphi]$ is the class of equivalence of $\varphi$ with respect to the equivalence relation $\sim$ in the set of all $k$-linear maps $Q \to K$,  where $\varphi_1 \sim \varphi_2$
      if $Im (\varphi_1 - \varphi_2) \subseteq k[[t]]$.

      \begin{theorem} \cite{Br-Gr1}, \cite{Br-Gr2}
    Let $A \to L \to Q$ be a short exact sequence of Lie algebras over a field $k$, where $L$ is finitely generated and both $A$ and $Q$ are abelian. Then the following conditions are equivalent :
    
      1. $A \wedge A$ is finitely generated as $U(Q)$-module via the diagonal $Q$-action i.e. $$q \circ ( a_1 \wedge a_2) = (q \circ a_1) \otimes a_2 + a_1 \otimes (q \circ a_2);$$  
      
      2. $A \otimes A$ is finitely generated as $U(Q)$-module via the diagonal $Q$-action;
      
    3. $L$ is a finitely presented Lie algebra;
    
      4. The Bryant-Groves invariant $\Delta(Q,A)$ does not have antipodal non-zero elements i.e. $\Delta(Q,A) \cap - \Delta(Q,A) = [0]$.  
      \end{theorem}

      \subsection{Finite presentation of metabelian Lie algebras : an example} The following example was first considered in \cite{Br-Gr1}. We quote it here to exhibit how the above criterion of finite presentability  is used.
      Suppose that $dim_k Q = 2$ and let $e_1$ and $e_2$ be a basis of $Q$ as a vector space.  Set $A  = k[x]$ the polynomial ring on one variable $x$. The adjoint action of $e_1$ is multiplication with $x$ and the adjoint action of $e_2$ is the multiplication with $x^r$ where $r \geq 1$ i.e.
      $$
      [e_1, a] = xa \hbox{ and } [e_2, a] = x^r a \hbox{ for } a \in A
      $$
      Thus $A$ is a cyclic $U(Q)$-module, so $R = A$ in this case. Every $k$-algebra homomorphism
      $$\widetilde{\varphi} : R = k[x] \to K$$
      is uniquely determined by $\widetilde{\varphi}(x) = \lambda \in K = \cup_{j \geq 1}  t^{-j} k[[t]]$. Then for $\varphi = \widetilde{\varphi} \mid_Q : Q \to K$ we have $\varphi(x) = \widetilde{\varphi} (e_1) = \lambda$ and $\varphi(e_2) = \widetilde{\varphi}(x^r) = \lambda^r$.
      
      Suppose that $[\varphi_1] \in \Delta(Q,A)$ be such that $[\varphi_2] \in \Delta(Q,A)$ and $[\varphi_2] = [- \varphi_1] \not= [0]$. Then $\varphi_i(x) = \lambda_i$ and $\lambda = \lambda_1 + \lambda_2 \in k[[t]], \lambda_1^r + \lambda_2^r \in k[[t]]$. 
      
      Thus if $r$ is odd we can choose $\lambda_1 \in K \setminus k[[t]]$, $\lambda_2 = - \lambda_1$. Thus $[\varphi_1] \in \Delta(Q,A) \cap - \Delta(Q,A)$, hence $\Delta(Q,A) \cap - \Delta(Q,A) \not= [0]$. 
      
      If $r$ is even,   $\lambda_1 \in a_0 t^{-s} + t^{-s+1} k[[t]]$ for some $a_0 \in k \setminus \{ 0 \}$. Then $\lambda_2 \in  -a_0 t^{-s} + t^{-s+1} k[[t]]$ and $\lambda_1^r + \lambda_2^r \in k[[t]]$, hence $2a_0^r = a_0^r + (- a_0)^r = 0$ in $k$. This is possible only if $char(k) = 2$.
 If $char(k) = 2$ $L$ is not finitely presented since $A$ is infinite dimensional, details can be found in \cite{Br-Gr1}.

      Hence by the Bryant-Groves criterion  we have : 
      
      1. if $r$ is odd then $L$ is not finitely presented.
      
      2. if $r$ is even, $char(k) \not= 2$ then $L$ is finitely presented;
      
      3. if $r$ is even, $char(k) = 2$ then $L$ is not finitely presented.

      \subsection{Preliminaries on metabelian Lie algebras of type $FP_m$}

      A Lie algebra $L$ is said to be of homological type $FP_m$ if the trivial $U(L)$-module has a projective resolution with finitely generated projectives in dimensions $ \leq m$, where $U(L)$ is the universal enveloping algebra of $L$.

      \medskip
      {\bf The $FP_m$-Conjecture for metabelian Lie algebras} 
   {\it  Let $A \to L \to Q$ be a short exact sequence of Lie algebras over a field $k$, where $L$ is finitely generated and both $A$ and $Q$ are abelian. Then the following conditions are equivalent :
    
    1. $L$ has homological type $FP_m$;
    
    2. $\wedge^m A$ is finitely generated as $U(Q)$-module via the diagonal $Q$-action i.e. $q \circ (a_1 \wedge \ldots \wedge a_m) = \sum_{1 \leq i \leq m} a_1 \wedge \ldots \wedge q \circ a_i \wedge \ldots a_m$
      
      3. If $[\varphi_i] \in \Delta_A(Q)$ for $ 1 \leq i \leq m$ then $Im (\sum_{1 \leq i \leq m} \varphi_i) \not\subseteq k[[t]]$.}
      
       \medskip

      \begin{theorem} \cite{Desi} \label{fp} Let $A \to L \to Q$ be a {\bf split} short exact sequence of Lie algebras over a field $k$, where $L$ is finitely generated and both $A$ and $Q$ are abelian. Then the $FP_m$-Conjecture for metabelian Lie algebras holds.
      \end{theorem} 
 
  \subsection{The lamplighter Lie algebra} \label{lamp}
   
   By definition the lamplighter group is $G = \mathbb{F}_2 \wr \mathbb{Z}$. It is a metabelian group of the type $ A \rtimes Q$ where $Q = \langle x \rangle \simeq \mathbb{Z}$ and $A$ is a cyclic free $\mathbb{F}_2 [Q]$-module, where $Q$ acts via conjugation. Hence $A \simeq \mathbb{F}_2[x^{\pm 1}]$ is the Laurent polynomial ring.	
   
 A Lie algebra version of $G$ is a Lie algebra $L$ over the finite field $\mathbb{F}_p$ that has an abelian ideal $A$ such that $dim(L/ A) = 1$ and $A$ is a cyclic left free $U(L/ A)$-module.  Thus $L = A  \leftthreetimes Q$ where $Q = \mathbb{F}_p b$. Let  $a$ be  a generator of $A$ as a left $U(L/A)$-module. Then $A \simeq U(L/A) \simeq \mathbb{F}_p[b]$ is a polynomial ring with variable $b$ and $[b,a]$ is the element  $b \in \mathbb{F}_p[b] = A$.
     
       \subsection{Metabelian Lie algebra of type $FP_n$ : an example} \label{FPmsection}
     \begin{lemma}  Let $Q$ be an abelian Lie algebra over a field $k$ with $dim_k Q = n$ and let $\{q_1, \ldots, q_n \}$ be a basis of $Q$ as a vector space. Suppose further that $char(k) = 0$ or $char(k) = p > n$. Consider the Lie algebra $L = A \leftthreetimes Q$ with $A$ an abelian Lie algebra, where   $A  = k[x]$ and the adjoint action of $q_i$ is multiplication with $x^{{i}}$  i.e.
      $
      [q_i, a] = x^{i}a. 
      $ Then
       $L$ is of type $FP_n$.
      \end{lemma}

      \begin{proof}
      Thus $A$ is a cyclic $U(Q)$-module, so $ 
      R = U(Q)/ ann_{U(Q)} A \simeq A$ in this case. Every $k$-algebra homomorphism
      $$\widetilde{\varphi} : R = k[x] \to K$$
      is uniquely determined by $\widetilde{\varphi}(x) = \lambda \in K = \cup_{j \geq 1}  t^{-j} k[[t]]$. Then for $\varphi = \widetilde{\varphi} \mid_Q : Q \to K$ we have $\varphi(x) = \widetilde{\varphi} (q_1) = \lambda$ and $\varphi(q_i) = \widetilde{\varphi}(x^{{i}}) = \lambda^{{i}}$.
      
      Suppose that $[\varphi_1], [\varphi_2] , \ldots, [\varphi_n] \in \Delta_A(Q)$ be such that $Im (\sum_{1 \leq i \leq n} \varphi_i) \in k[[t]]$.  Set 
      $$\widetilde{\varphi}_i(x) = \lambda_i \in \cup_{j \geq 1} t^{-j} k[[t]].$$
      Then
      \begin{equation} \label{FPm}
      \lambda_1^{{j}} + \ldots + \lambda_n^{{j}} \in k[[t]] \hbox{ for } 1 \leq j \leq n.\end{equation}
      Consider a height function
      $$
      h : K = \cup_{j \geq 1} t^{-j} k[[t]] \to \mathbb{Z}
      $$
      sending $f \in a t^{-j} + t^{-j+1}k[[t]]$ for some $a \in k \setminus \{ 0 \}$ to $-j$. 
      Let
      $$j_0 = min \{ height(\lambda_i) \}_{1 \leq i \leq n}
      $$
      and
      $$
      \lambda_i  \in  b_i t^{j_0} + t^{j_0 + 1} k[[t]]
      $$
      for some $b_i \in k$.    Suppose $j_0 < 0$. Then (\ref{FPm}) implies 
      $$
      \sum_{1 \leq i \leq n} b_i^{{j}} = 0 \hbox{ for all } 1 \leq j \leq n.
      $$ Then {\it if $char(k) = 0$ or $char(k) = p > n$} we deduce that $b_1 = \ldots = b_n = 0$, a contradiction.
      Hence $j_0 \geq 0$, so $[\varphi_i] = [0]$ for all $ 1 \leq i \leq n$.
By Theorem \ref{fp}  $L$ is of type $FP_m$ if $char(k) = 0$ or $char(k) = p > n$.
\end{proof}
\subsection{On the self-similarity of the above example : proof of Theorem D}

\begin{theorem} Let $L$ be the metabelian Lie algebra from Section \ref{FPmsection} without restriction on $char(k)$. Then $L$ is  a  contracting, faithful self-similar metabelian  Lie algebra that is not regularly weakly branched. If $char(k) = p > 0$ then $L$ is a transitive self-similar Lie algebra.
\end{theorem}

\begin{proof}
We consider the case when $ char(k) = p > 0$. The case when $char(k) = 0$ is the same, in this case we cannot define the notion of transitive self-similar Lie algebra.

\medskip
1. {\it The self-similarity of $L$}

We consider the example above with $char(k) =p > 0 $  and write as generators $q_1, \ldots, q_n \in Q$ and $a \in A$ ( here $a = 1$). Now we construct virtual endomorphism
$$
\theta : H =  A_0 \leftthreetimes Q \to L = A \leftthreetimes Q
$$
 that is the identity on $Q$ and by definition $A_0 = x k[x]$ and
 $$
 \theta(a_0 ) = a_0 / x.
 $$
 Then $H$ is an ideal of $L$ such that $dim(L/H) = 1$. 
 Furthermore $$L = H  \leftthreetimes L_0, \hbox{ where } L_0 = k a.$$ We identify $a$ with the derivation $\partial / \partial x$ i.e. $\psi(a) = \partial/ \partial x$. Thus we can apply Theorem B for $n = 1$ provided
 $$
 \theta(a^p \circ h) = 0
 $$
 for every $h \in H$. But for $h \in H$ we have $a^2 \circ h = a \circ (a \circ h) = [a, [a,h]] \subseteq [A,A] = 0$. Then by Theorem A we have for $h \in H$
\begin{equation} \label{formula} 
\psi(h) = 1 \otimes \theta(h) + x \otimes \theta(a \circ h),
\end{equation}
where $a \circ h = [a,h]$.

\medskip
2. {\it Faithfulness}

It is easy to see that $L$ is a faithful $\psi$-self-similar Lie algebra using Theorem A c).
Indeed suppose that $J$ is an ideal of $L$ such that $\theta(J) \subseteq J$ and $J \subseteq H$. Then $J_0 : = J \cap A \subseteq H \cap A = A_0$ and $\theta(J_0) \subseteq \theta(J) \cap \theta(A_0) \subseteq J \cap A = J_0.
$ Thus $J_0 \subseteq \cap_{i \geq 1} A x^i = 0$, so $J \cap A = 0$. Then $[J, A] \subseteq J \cap A = 0$, so $[J + A , A] \subseteq [J, A] + [A,A] = 0$.
 Then for the canonical map $\pi : L \to L/ A = Q$ we have that $\pi(J) \subseteq C_{Q} (A) = \{ q \in Q \mid [q, A]= 0 \} = 0$.  Since $A \cap J = 0$ we deduce that $J \simeq \pi(J) = 0$.  Since each $\theta$-invariant ideal in $L$ that is contained in $H$ is zero we deduce by Theorem A c) that $L$ is a faithful $\psi$-self-similar Lie algebra.
 
 \medskip
 3. {\it Recurrency}
 
 By construction $\theta$ is surjective and by definition this means that $L$ is recurrent self-similar Lie algebra.

 \medskip
 4. {\it Finite state generators}
 
    Consider the generating set $\{q_1, \ldots, q_n , a \}$ of $L$. Since $\psi(a)  = \partial / \partial x \in {\mathcal Der} X$ we have that  $a$ is a finite state element of $L$.    
    By (\ref{formula})
    $$
    \psi(q_i) = 1 \otimes q_i + x \otimes ([a, q_i]/ x) = 1 \otimes q_i - x \otimes a_{i-1}  
    $$
    where $a_{j} = x^j \in A$.
    If $i \geq 2$ then
    $$
    \psi ( a_{i-1}) = 1 \otimes \theta (a_{i-1}) = 1 \otimes a_{i-2} 
    $$
    and for $i = 1$, $a_0 = a$
    $$
    \psi(a) = \partial / \partial x \in {\mathcal Der}X.
    $$
    Then for $S_i$ the vector space over $k$ with basis $\{ q_i, a_{j} \}_{0 \leq j \leq i-1}$ we have
    $$
    \psi(S_i) \subseteq (X \otimes S_i ) \leftthreetimes {\mathcal Der} X.
    $$
    Since $q_i \in S_i$ we deduce that $q_i$ is a finite state element of $L$.

\medskip
5. {\it Transitivity}

It follows directly from Theorem A e).
  
  \medskip
  6. {\it Contractibility}
  
 Let  $S$ be the vector subspace of $L$ generated by $Q$ and $\{ a_j = x^j \mid 0 \leq j \leq n-1 \} \subset A$.
By item 4. above we have that for every $m \geq 1$  and $q \in Q$
$$S^m(q) \subseteq S.
$$

On other hand for $b \in x^i k[x] \setminus x^{i-1} k[x]$ for some $i \geq 1$  we have
$$
\psi(b) = 1 \otimes (b/x) \in X \otimes L \hbox{ and } \psi(a) \in {\mathcal Der} X
$$
hence
$$
S^{i+1}(b) = S(a) = 0.
$$
 Then for every $l \in L$ there is $m_0 = m_0 (l)$ such that for every $m \geq m_0$ we have $S^m(l) \subseteq S$.
 
 \medskip
7. {\it Not regularly weakly branched}

By \cite{Bartholdi} a recurrent transitive self-similar Lie algebra $L$ is regularly weakly branched if there is a non-zero ideal $K$ of $L$ such that $\psi(K)$ contains $X \otimes K$.
Assume that such ideal $K$ exists. 

Let $k_0 \in K$. Then $(1 + x) \otimes k_0 \in X \otimes K \subseteq \psi(K)$, so there is $k_1 \in K$ such that
$$
(1 + x) \otimes k_0 = \psi(k_1).
$$
Since $\psi(k_1) \in X \otimes L$ we have that $k_1 \in H$ and by (\ref{formula})
$$
1 \otimes k_0  + x \oplus k_0 = (1 + x) \otimes k_0  = \psi(k_1) = 1 \otimes \theta (k_1) + x \otimes \theta([a,k_1]). 
$$
 Thus
 $$ \theta(k_1) = k_0 = \theta([a, k_1]) \in \theta(A \cap H) \subseteq A$$ 
 and so $k_0 \in A$. Since $k_0$ is an arbitrary element of $K$ we deduce that
 $$K \subseteq A.$$
 By (\ref{formula}) we have
 $$\psi(A) = \psi (A \cap H) + \psi(k a)  = 1 \otimes \theta(A \cap H) + k \partial / \partial x =    1 \otimes A + k \partial / \partial x,$$ hence
 $$\psi(A) \cap (x \otimes L) = 0.$$
 Then
 $$x \otimes K = (x \otimes K) \cap \psi(K) \subseteq (x \otimes L) \cap \psi(A) = 0,
 $$
 so $K = 0$, a contradiction.
   Thus $L$ is not regularly weakly branched as a $\psi$-self-similar Lie algebra.
 \end{proof}

\end{document}